\documentclass[10pt,reqno]{amsart}
\usepackage{amsfonts}
\usepackage{amsmath,amsthm,amssymb,amsfonts,amscd}
\usepackage{mathrsfs}
\usepackage{lipsum}
\usepackage{bbding}
\usepackage{graphicx,latexsym}
\usepackage{hyperref}

\newcommand{\bea}{\begin{eqnarray}}
\newcommand{\eea}{\end{eqnarray}}
\newcommand{\bna}{\begin{eqnarray*}}
\newcommand{\ena}{\end{eqnarray*}}

\numberwithin{equation}{section}

\renewcommand{\thefootnote}{\fnsymbol{footnote}}
\theoremstyle{plain}
\newtheorem{theorem}{Theorem}
\newtheorem{lemma}{Lemma}

\newtheorem{proposition}[lemma]{Proposition}

\theoremstyle{definition}

\newcommand{\supp}{\operatorname{supp}}
\newcommand\blfootnote[1]{%
  \begingroup
  \renewcommand\thefootnote{}\footnote{#1}%
  \addtocounter{footnote}{-1}%
  \endgroup
}

\renewcommand{\Re}{\operatorname{Re}}

\begin{document}

\title{A bound for twists of $\rm GL_3\times GL_2$ $L$-functions with composite modulus}

		\author{Qingfeng Sun}
	\date{}
	\address{School of Mathematics and Statistics, Shandong University, Weihai\\Weihai, Shandong
		264209, China}
	\email{qfsun@sdu.edu.cn}

	\author{Yanxue Yu}
	\date{}
	\address{School of Mathematics and Statistics, Shandong University, Weihai\\Weihai, Shandong
		264209, China}
    \email{yanxueyu@mail.sdu.edu.cn}

\begin{abstract}
Let $\pi$ be a Hecke-Maass cusp form for $\rm SL_3(\mathbf{Z})$ and
let $g$ be a holomorphic or Maass cusp form
for $\rm SL_2(\mathbf{Z})$. Let $\chi$ be
a primitive Dirichlet character of modulus $M=M_1M_2$ with $M_i$ prime, $i=1,2$.
Suppose that $M^{1/2+2\eta}<M_1<M^{1-2\eta}$ with $0<\eta<1/8$. Then we have
$$
L\left(\frac{1}{2},\pi\otimes g \otimes \chi\right)\ll_{\pi,g,\varepsilon}
M^{3/2-\eta+\varepsilon}.
$$
\end{abstract}
\thanks{Q. Sun was partially
  supported by the National Natural Science Foundation
  of China (Grant Nos. 11871306 and 12031008)}



\maketitle


\section{Introduction}

Let $\pi$ be a Hecke-Maass cusp form for $\rm SL_3(\mathbf{Z})$ with normalized Fourier
coefficients $\lambda_{\pi}(n_1,n_2)$ such that $\lambda_{\pi}(1,1)=1$.
Let $g$ be a holomorphic or Maass cusp form for $\rm SL_2(\mathbf{Z})$
with normalized Fourier coefficients $\lambda_g(n)$ such that $\lambda_g(1)=1$.
Let $\chi$ be a primitive
Dirichlet character modulo $M$. The $L$-function attached to the twisted form
$\pi\otimes g\otimes \chi$ is given by the Dirichlet series
\bna
L(s,\pi \otimes g\otimes\chi)=\sum_{r=1}^{\infty}\sum_{n=1}^{\infty}
\frac{\lambda_{\pi}(n,r)\lambda_g(n)\chi(n)}{(r^2n)^s}
\ena
for Re$(s)>1$, which can be continued to an entire function with
a functional equation of arithmetic conductor $M^6$ (ignoring the dependence
on the form $\pi$ and $g$). Then by the Phragmen-Lindel\"{o}f convexity
principle one derives the convexity bound
$L\left(1/2,\pi \otimes g\otimes\chi\right)\ll_{\pi,g,\varepsilon}
M^{3/2+\varepsilon}$,
where hereafter $\varepsilon$ denotes arbitrary small positive constant, which is not
necessarily the same at different occurrences.
The important challenge for us is to prove a subconvexity bound which improves the convexity bound
by reducing the exponent $3/2$. For $\pi$ self-dual, $g$ a Hecke-Maass csup form for
$\Gamma_0(M)\subseteq \rm SL_2(\mathbf{Z})$, $M$ prime and $\chi$ a
quadratic character modulo $M$, Blomer \cite{B} first proved the subconvexity
$$
L\left(1/2,\pi \otimes g\otimes\chi\right)\ll_{\pi,g,\varepsilon}
M^{5/4+\varepsilon}.
$$
Later, Sharma \cite{Sharma2019} generalized Blomer's result to any $\rm GL_3$ form and obtained the
bound $O_{\pi,g,\varepsilon}\left(M^{3/2-1/16+\varepsilon}\right)$ for $M$ prime. Recently, Lin, Michel and
Sawin \cite{LMS}
extended this bound by replacing $\chi$ by a generic trace function
$K: (\mathbf{Z}/M\mathbf{Z})^{\times}\rightarrow \mathbf{C}$ which is the
Frobenius trace function associated to some geometrically
irreducible middle extension sheaf $\mathcal{F}$  on $\mathbf{P}^1_{\mathbf{F}_M}$
pure of weight $0$ satisfying some additional generic conditions.

In this paper, we further extend the previous works by considering
the case that $\chi$ is a Dirichlet character of composite modulus.
Our main result is the following.

\begin{theorem}\label{main theorem}
Let $\pi$ be a Hecke-Maass cusp form for $\rm SL_3(\mathbf{Z})$ and
let $g$ be a holomorphic or Maass cusp form
for $\rm SL_2(\mathbf{Z})$. Let $\chi=\chi_1\chi_2$ be
a Dirichlet character with $\chi_i$ primitive modulo $M_i$.
Suppose that $M_1$, $M_2$ are primes such that $M^{1/2+2\eta}<M_1<M^{1-2\eta}$,
where $M=M_1M_2$ and $0<\eta<1/8$. Then we have
$$
L\left(\frac{1}{2},\pi\otimes g \otimes \chi\right)\ll_{\pi,g,\varepsilon}
M^{3/2-\eta+\varepsilon}.
$$
\end{theorem}

This paper is inspired by the work of
Munshi \cite{Munshi1} where the subconvexity problem of the $\rm GL_3$
$L$-function $L\left(s,\pi\otimes \chi\right)$ with
the Dirichlet character $\chi$ of
modulus a composite number is considered. More precisely,
for $\chi=\chi_1\chi_2$ a Dirichlet character with $\chi_i$ primitive
modulo prime $M_i$ such that $\sqrt{M_2}M^{4\vartheta}<M_1<M_1M^{-3\vartheta}$,
Munshi proved that
$$
L\left(\frac{1}{2},\pi\otimes \chi\right)\ll_{\pi,\varepsilon}
M^{3/4-\vartheta+\varepsilon},
$$
where $M=M_1M_2$ and $0<\vartheta<1/28$. As in \cite{Munshi1}, for the sake of brevity, we do not pursue
the best $\eta$ here.

Let's take a quick look at the method. We follow Munshi's strategy in \cite{Munshi1} and start the proof by
using the approximate functional equation to translate the problem as
\bna
L\left(1/2,\pi\otimes g \otimes \chi\right)
\ll N^{-1/2}\sum_{r^2n\sim N}\lambda_{\pi}(n,r)\lambda_g(n)\chi(n)
\ena
with $N\sim M^3$.
Since the contribution from large $r$ can be trivially controlled by Cauchy-Schwarz and
Rankin-Selberg's estimate, and the parameter $r$ in our actual proof is fixed, without loss
of generality, we only consider the contribution from $r=1$ here and denote
\bna
S(N)=\sum_{n\sim N}\lambda_{\pi}(n,1)\lambda_g(n)\chi(n).
\ena
Next, we separate the oscillations from $\lambda_{\pi}(n,1)$ and $\lambda_g(n)\chi(n)$
by applying the $\delta$-method due to Duke, Friedlander and Iwaniec \cite{DFI} together with
a conductor lowering mechanism developed in a series of works of Munshi \cite{Munshi1},
\cite{Munshi2} and \cite{Munshi2018}; see Lemma \ref{circle method}. Then we get
\bna
S(N)\approx \frac{1}{QM_1}\sum_{q\sim Q\atop (q,M)=1}\frac{1}{q}\;
   \sum\limits_{a(\text{mod} \,qM_1)\atop (a,qM_1)=1}\mathop{\sum\sum}_{n,m\sim N}
  \lambda_{\pi}(n,1)\lambda_g(m)\chi(m)e\left(\frac{a(m-n)}{qM_1}\right),
\ena
where we take $Q=\sqrt{N/M_1}$. This choice of $Q$ benefits from
the conductor lowering trick, where we write $\delta(n)$ as $\delta(n/M_1)\mathbf{1}_{M_1|n}$.
Now the $m$-and $n$-sums are in a form ready to use the
Voronoi summation formulas. By employing the $\rm GL_2$ and $\rm GL_3$-Voronoi formulas
to the $m$-and $n$-sums respectively, we arrive at an expression of the form
(in fact there is another smaller sum which we omit here for simplicity)
\bna
\sum_{n_1^2n_2\sim N^2/Q^3}\lambda_{\pi}(n_1,n_2)\sum_{q\sim Q}\;
\sum_{m\sim M^2/M_1}\quad\lambda_g(m)
C^*(n_1,n_2,m,q),
\ena
where the character sum can be factored as sub-character sums with modulus $qr/n_1$ and $M_1$, respectively.
Next we apply Cauchy-Schwartz inequality to get rid of the Fourier coefficients $\lambda_{\pi}(n_1,n_2)$.
Then we need to deal with
\bna
\sum_{n_1^2n_2\sim N^2/Q^3}\bigg|\sum_{q\sim Q}
\sum_{m\sim M^2/M_1}\lambda_g(m)
C^*(n_1,n_2,m,q)\bigg|^2.
\ena
Opening the square and applying Poisson summation to the
sum over $n_2$, we end up with a character sum which has essentially been
estimated in \cite{LMS} and \cite{LS}. There is certainly a complex integral
after Poisson, which we evaluate by the stationary phase analysis and the
two dimensional second derivative test; see Lemma \ref{integral estimate}.

\section{Proof of main theorem}
By applying the approximate functional equation of $L(s,\pi \otimes g \otimes \chi)$
(Theorem 5.3 and Proposition 5.4 of \cite{IK}), we have
\bea\label{app fun}
    L(1/2,\pi\otimes g \otimes \chi) \ll
    M^\varepsilon \sup_{N \ll M^{3+\varepsilon}} \frac{|\mathcal{S}(N)|}{\sqrt{N}} + M^{-A},
\eea
for any $A>0$, where
\bna
\mathcal{S}(N) = \sum_{r\geq1} \sum_{n\geq1} \lambda_{\pi}(n,r)\lambda_g(n)
    \chi(n)V(r^2n/N)
\ena
with some compactly supported smooth function $V$ such that
$\supp V\subset [1,2]$ and  $V^{(j)}(x)\ll_j 1$ for any integer $j\geq 0$.
By Cauchy-Schwarz inequality and the Rankin-Selberg estimates
(see Molteni \cite[Theorem 2]{Mol})
\bea\label{GL2 RS}
\sum_{n\ll x}|\lambda_g(n)|^2\ll_{g,\varepsilon} x^{1+\varepsilon}
\eea
and
\bea\label{GL3 RS}
\sum_{nr^2\ll x}|\lambda_\pi(n,r)|^2 \ll_{\pi,\varepsilon} x^{1+\varepsilon},
\eea
we have the trivial bound $ \mathcal{S}(N) \ll N$.
Thus the contribution from $ N \leq M^{3-\theta} $
to $ L(1/2, \pi\otimes g \otimes \chi ) $ in (2.1) is
at most $O\left(M^{(3-\theta)/2}\right)$,
where $ \theta > 0$ is a constant which will be chosen optimally later.
Moreover, the contribution from ``large" $r$ is
\bna
&&\sum_{r \geq M^\theta}\sum_{n \ll N/r^2}
\lambda_\pi(n,r)\lambda_g(n)\chi(n)V(r^2n/N)\\
&\ll& \bigg(\sum_{r \geq M^\theta}r^2\bigg|\sum_{n \ll N/r^2}
\lambda_\pi(n,r)\lambda_g(n)\chi(n)V(r^2n/N)\bigg|^2
\bigg)^{\frac{1}{2}}
\bigg(\sum_{r \geq M^\theta}\frac{1}{r^2}\bigg)^{\frac{1}{2}}\\
&\ll& \bigg(\sum_{r \geq M^\theta}r^2\sum_{n \ll N/r^2}
\bigg|\lambda_\pi(n,r)\bigg|^2
\sum_{n \ll N/r^2}\bigg|\lambda_g(n)\bigg|^2\bigg)^{\frac{1}{2}}
M^{-\frac{\theta}{2}}\\
&\ll& N^{\frac{1}{2}}M^{-\frac{\theta}{2}}
\bigg(\sum_{nr^2 \ll N}\bigg|\lambda_\pi(n,r)\bigg|^2
\bigg)^{\frac{1}{2}}\\
&\ll& NM^{-\frac{\theta}{2}}.
\ena
It follows that the contribution from $ r \geq M^\theta$
to $ L(1/2, \pi \otimes g \otimes \chi)$ in \eqref{app fun} is bounded by
$O\left(M^{(3-\theta)/2+\varepsilon}\right)$. Assembling the above argument, we obtain
\bea\label{set up}
L\left(\frac{1}{2}, \pi \otimes g \otimes \chi\right)
\ll_{\pi, g, \epsilon}M^\epsilon
\sum_{r\leq M^\theta}
\frac{1}{r}\sup_{M^{3-\theta}/r^2\leq N \leq M^{3+\epsilon}/r^2}
\frac{|\mathcal{S}_r(N)|}{\sqrt N}
+M^{\frac{3-\theta}{2}+\varepsilon},
\eea
where
\bea\label{initial sum}
\mathcal{S}_r(N)=\sum_{n=1}^{\infty}\lambda_{\pi}(n,r)\lambda_g(n)\chi(n)V\left(n/N\right).
\eea
The contribution from the terms with $(r,M)\neq 1$ is smaller. For simplicity, we assume
\bea\label{asummption on r}
r<\min\{M_1,M_2\}.
\eea
Then $(r,M_1 M_2 )=1$.
Under the assumption \eqref{asummption on r} we will establish the following.
\begin{proposition}\label{prop1}
We have
\bna
\sum_{r\leq M^\theta}
\frac{1}{r}\sup_{M^{3-\theta}/r^2\leq N \leq M^{3+\epsilon}/r^2}
\frac{|\mathcal{S}_r(N)|}{\sqrt N}\ll
M^{3/4}M_1^{3/4}+M^{1+\theta/2}M_1^{1/4}
+M^{3/2}M_1^{-1/4}+M^{7/4}M_1^{-1/2}.
\ena
\end{proposition}

By Proposition \ref{prop1} the right hand side of \eqref{set up}
can be dominated by $M^{(3-\theta)/2+\varepsilon}$ if
\bna
M^{1/2+\theta}<M_1<M^{1-2\theta/3}, \qquad 0<\theta<3/10.
\ena
Then by combining the assumption in \eqref{asummption on r},
one sees that
\bna
L\left(1/2, \pi \otimes g \otimes \chi\right)
\ll_{\pi, g, \epsilon}M^{\frac{3-\theta}{2}+\varepsilon}
\ena
for
\bna
M^{1/2+\theta}<M_1<M^{1-\theta}, \qquad 0<\theta<1/4.
\ena
So Theorem \ref{main theorem} follows.
The following sections of this paper will be devoted to the proof of Proposition \ref{prop1}.

\subsection{Applying the $\delta$-method }
Define $\delta: \mathbf{Z}\rightarrow \{0,1\}$ with
$\delta(0)=1$ and $\delta(n)=0$ for $n\neq 0$.
To separate oscillations from $\lambda_\pi(r,n)$ and $\lambda_g(n)\chi(n)$,
we use a version of the delta method due to Duke, Freidlander
and Iwaniec (see \cite[Chapter 20]{IK})
which states that for any $n\in \mathbf{Z}$ and $Q\in \mathbb{R}^+$, we have
\bea\label{DFI's}
\delta(n)=\frac{1}{Q}\sum_{1\leq q\leq Q}\frac{1}{q}\;
\sideset{}{^*}\sum_{a(\text{{\rm mod }} q)}
e\left(\frac{na}{q}\right)\int_{\mathbb{R}}
\omega(q,\zeta)e\left(\frac{n\zeta}{qQ}\right)\mathrm{d}\zeta,
\eea
where the $*$ on the sum indicates
that the sum over $a$ is restricted to $(a,q)=1$.
The function $\omega$ has the following properties (see (20.158) and (20.159)
of \cite{IK} and Lemma 15 of \cite{HB})
\bea\label{omega-h}
\omega(q,\zeta)\ll |\zeta|^{-A}, \;\;\;\;\;\; \omega(q,\zeta) =1+h(q,\zeta)\;\;\text{with}\;\;h(q,\zeta)=
O\left(\frac{Q}{q}\left(\frac{q}{Q}+|\zeta|\right)^A\right)
\eea
for any $A>1$ and
\bea\label{rapid decay omega}
\zeta^j\frac{\partial^j}{\partial \zeta^j}\omega(q,\zeta)\ll
(\log Q)\min\left\{\frac{Q}{q},\frac{1}{|\zeta|}\right\}, \qquad j\geq 1.
\eea

\begin{lemma}\label{circle method}
Let $M=M_1M_2$ with $M_i$ prime, $i=1,2$.
Then
\bna
\delta(n)&=& \sum_{s=0,1}\;
\sum_{\ell=0}^{[\log Q/\log M_2]}
\frac{1}{Q}
\sum_{q\leq Q/M_2^\ell \atop (q,M)=1}
\frac{1}{q M_1M_2^\ell}\;
\sideset{}{^*}\sum_{a(\text{{\rm mod }} q M_1^{1-s} M_2^\ell)}
e\left(\frac{an}{q M_1^{1-s} M_2^\ell}\right)\qquad\qquad\\
&\times&
\int_{\mathbb{R}} U\left(\frac{\zeta}{M^\varepsilon}\right)\;
  \omega(q M_2^\ell,\zeta)\;
  e\left(\frac{n\zeta}{qQM_1M_2^\ell}\right)\mathrm{d}\zeta\\
&+&\sum_{\ell=0}^{[\log Q/\log M_2]}\;
\sum_{t=1}^{[\log Q/\log M_1]}
\frac{1}{Q}
\sum_{q\leq Q/M_1^t M_2^{\ell}\atop (q,M)=1}\frac{1}{qM_1^{1+t}M_2^{\ell}}\;
\sideset{}{^*}\sum_{a(\text{{\rm mod }} qM_1^{1+t}M_2^{\ell})}
e\left(\frac{an}{qM_1^{1+t}M_2^{\ell}}\right)\qquad\qquad\\
&\times&
\int_{\mathbb{R}}U\left(\frac{\zeta}{M^\varepsilon}\right)
  \omega(q M_1^t M_2^{\ell},\zeta)\;
  e\left(\frac{n\zeta}{q QM_1^{1+t}M_2^{\ell}}\right)
  \mathrm{d}\zeta+O\left(M^{-A}\right)
\ena
for any $A>0$, where $U(x)\in C_c^{\infty}(-2,2)$ is a smooth positive function satisfying
$U(x) =1$ if $x\in [-1,1]$ and $U^{(j)}(x)\ll_j 1$ for any integer $j\geq 0$.
\end{lemma}

\begin{proof}
Note that the contribution from $|\zeta|\leq M^{-G}$  in \eqref{DFI's}
is negligible for $G>0$ sufficiently large and by the first property in \eqref{omega-h}, we can restrict $\zeta$ in the range
$|\zeta|\leq M^{\varepsilon}$ up to an negligible error. So we can
insert a smooth partition of unity for the $\zeta$-integral in \eqref{DFI's} and write
\bna
\delta(n)=\frac{1}{Q}\sum_{q\leq Q} \;\frac{1}{q}\;
\sideset{}{^*}\sum_{a\bmod{q}}e\left(\frac{na}{q}\right)
\int_\mathbb{R}U\left(\frac{\zeta}{M^\varepsilon}\right)
\omega(q,\zeta) e\left(\frac{n\zeta}{qQ}\right)\mathrm{d}\zeta+O_A(M^{-A})
\ena
for any $A>0$, where $U(x)\in \mathcal{C}_c^{\infty}(1,2)$
satisfying $U^{(j)}(x)\ll_j 1$ for any integer $j\geq 0$.
Define $\mathbf{1}_\mathscr{F}=1$ if $\mathscr{F}$ is true, and is 0 otherwise.
Following Munshi \cite{Munshi1} we write $\delta(n)$ as
$\delta(n/M_1)\mathbf{1}_{M_1|n}$ and detect the congruence by additive
characters to get
\bea\label{delta}
\delta(n)
=\frac{1}{QM_1}\sum_{q\leq Q}\frac{1}{q}
\sum_{b(\text{{\rm mod }} M_1)}\;
\sideset{}{^*}\sum_{a(\text{{\rm mod }} q)}e\left(n\frac{a+bq}{qM_1}\right)
\int_{\mathbb{R}}U\left(\frac{\zeta}{M^\varepsilon}\right)
  \omega(q,\zeta)e\left(\frac{n\zeta}{qQM_1}\right)\mathrm{d}\zeta+O_A(M^{-A}),
\eea
where the first term can be further written as $\delta_1(n)+\delta_2(n)$ with
\bna
\delta_1(n)&=&\frac{1}{QM_1}\sum_{q\leq Q\atop (q,M_1)=1}\frac{1}{q}
\sum_{b(\text{{\rm mod }} M_1)}\;
\sideset{}{^*}\sum_{a(\text{{\rm mod }} q)}e\left(n\frac{a+bq}{qM_1}\right)
\int_{\mathbb{R}}U\left(\frac{\zeta}{M^\varepsilon}\right)
  \omega(q,\zeta)e\left(\frac{n\zeta}{qQM_1}\right)\mathrm{d}\zeta,\\
\delta_2(n)&=&\frac{1}{QM_1}\sum_{q\leq Q\atop M_1|q}\frac{1}{q}
\sum_{b(\text{{\rm mod }} M_1)}\;
\sideset{}{^*}\sum_{a(\text{{\rm mod }} q)}e\left(n\frac{a+bq}{qM_1}\right)
\int_{\mathbb{R}}U\left(\frac{\zeta}{M^\varepsilon}\right)
  \omega(q,\zeta)e\left(\frac{n\zeta}{qQM_1}\right)\mathrm{d}\zeta.
\ena
For $\delta_1(n)$, making a change of variable $a\to aM_1$, we have
\bna
\delta_1(n)
&=&\frac{1}{QM_1}\sum_{q\leq Q\atop (q,M_1)=1}\frac{1}{q}
\sum_{b(\text{{\rm mod }} M_1)}\;
\sideset{}{^*}\sum_{a(\text{{\rm mod }} q)}e\left(n\frac{aM_1+bq}{qM_1}\right)
\int_{\mathbb{R}}U\left(\frac{\zeta}{M^\varepsilon}\right)
  \omega(q,\zeta)e\left(\frac{n\zeta}{qQM_1}\right)\mathrm{d}\zeta\\
&=&\frac{1}{QM_1}\sum_{q\leq Q\atop (q,M_1)=1}\frac{1}{q}\;
\sideset{}{^*}\sum_{b(\text{{\rm mod }} M_1)}\;
\sideset{}{^*}\sum_{a(\text{{\rm mod }} q)}e\left(n\frac{aM_1+bq}{qM_1}\right)
\int_{\mathbb{R}}U\left(\frac{\zeta}{M^\varepsilon}\right)
  \omega(q,\zeta)e\left(\frac{n\zeta}{qQM_1}\right)\mathrm{d}\zeta\\
&&+\frac{1}{QM_1}\sum_{q\leq Q\atop (q,M_1)=1}\frac{1}{q}\;
\sideset{}{^*}\sum_{a(\text{{\rm mod }} q)}e\left(\frac{an}{q}\right)
\int_{\mathbb{R}}U\left(\frac{\zeta}{M^\varepsilon}\right)
  \omega(q,\zeta)e\left(\frac{n\zeta}{qQM_1}\right)\mathrm{d}\zeta.
\ena
Observe that in the first sum of the last expression, as $a$ varies over a set of representatives of
the reduced residue classes modulo $q$ and $b$ varies over a set of representatives of the reduced residue
classes modulo $M_1$, $aM_1 + bq$ varies over a set of representatives of the reduced
residue classes modulo $qM_1$. Then
\bea\label{delta 1}
\delta_1(n)&=&\frac{1}{QM_1}\sum_{q\leq Q\atop (q,M_1)=1}\frac{1}{q}
\;\sideset{}{^*}\sum_{a(\text{{\rm mod }} qM_1)}
e\left(\frac{an}{qM_1}\right)
\int_{\mathbb{R}}U\left(\frac{\zeta}{M^\varepsilon}\right)
  \omega(q,\zeta)e\left(\frac{n\zeta}{qQM_1}\right)\mathrm{d}\zeta\nonumber\\
&&+\frac{1}{QM_1}\sum_{q\leq Q\atop (q,M_1)=1}\frac{1}{q}\;
\sideset{}{^*}\sum_{a(\text{{\rm mod }} q)}e\left(\frac{an}{q}\right)
\int_{\mathbb{R}}U\left(\frac{\zeta}{M^\varepsilon}\right)
  \omega(q,\zeta)e\left(\frac{n\zeta}{qQM_1}\right)\mathrm{d}\zeta.
\eea
For $\delta_2(n)$, similarly making a change of variable $q\to qM_1$, we have
\bea\label{delta 2}
\delta_2(n)&=&\frac{1}{QM_1^{2}}\sum_{q\leq Q/M_1}\frac{1}{q}
\sum_{b(\text{{\rm mod }} M_1)}\;
\sideset{}{^*}\sum_{a(\text{{\rm mod }} qM_1)}e\left(n\frac{a+bM_1q}{qM_1^{2}}\right)
\int_{\mathbb{R}}U\left(\frac{\zeta}{M^\varepsilon}\right)
  \omega(M_1q,\zeta)e\left(\frac{n\zeta}{qQM_1^{2}}\right)\mathrm{d}\zeta\nonumber\\
&=&\frac{1}{QM_1^{2}}\sum_{q\leq Q/M_1\atop (q,M_1)=1}\frac{1}{q}
\sum_{b(\text{{\rm mod }} M_1)}\;
\sideset{}{^*}\sum_{a(\text{{\rm mod }} qM_1)}e\left(n\frac{a+bM_1q}{qM_1^{2}}\right)
\int_{\mathbb{R}}U\left(\frac{\zeta}{M^\varepsilon}\right)
  \omega(M_1q,\zeta)e\left(\frac{n\zeta}{qQM_1^{2}}\right)\mathrm{d}\zeta\nonumber\\
&&+\frac{1}{QM_1^{3}}\sum_{q\leq Q/M_1^2}\frac{1}{q}
\sum_{b(\text{{\rm mod }} M_1)}\;
\sideset{}{^*}\sum_{a(\text{{\rm mod }} qM_1^2)}e\left(n\frac{a+bM_1^2q}{qM_1^3}\right)
\int_{\mathbb{R}}U\left(\frac{\zeta}{M^\varepsilon}\right)
  \omega(M_1^2q,\zeta)e\left(\frac{n\zeta}{qQM_1^3}\right)\mathrm{d}\zeta\nonumber\\
&=&\sum_{t=1}^{[\log Q/\log M_1]}
\frac{1}{Q}\sum_{q\leq Q/M_1^t\atop (q,M_1)=1}\frac{1}{qM_1^{1+t}}
\sideset{}{^*}\sum_{a(\text{{\rm mod }} qM_1^{1+t})}
e\left(\frac{an}{qM_1^{1+t}}\right)
\int_{\mathbb{R}}U\left(\frac{\zeta}{M^\varepsilon}\right)
  \omega(M_1^tq,\zeta)e\left(\frac{n\zeta}{qQM_1^{1+t}}\right)\mathrm{d}\zeta.\nonumber\\
\eea
By \eqref{delta}-\eqref{delta 2}, we can therefore write
\bna
\delta(n)&=& \sum_{s=0,1}\frac{1}{Q}\sum_{q\leq Q\atop (q,M_1)=1}\frac{1}{qM_1}
\;\sideset{}{^*}\sum_{a(\text{{\rm mod }} qM_1^{1-s})}
e\left(\frac{an}{qM_1^{1-s}}\right)
\int_{\mathbb{R}}U\left(\frac{\zeta}{M^\varepsilon}\right)
  \omega(q,\zeta)e\left(\frac{n\zeta}{qQM_1}\right)\mathrm{d}\zeta\nonumber\\
&&+\sum_{t=1}^{[\log Q/\log M_1]}
\frac{1}{Q}\sum_{q\leq Q/M_1^t\atop (q,M_1)=1}\frac{1}{qM_1^{1+t}}
\sideset{}{^*}\sum_{a(\text{{\rm mod }} qM_1^{1+t})}
e\left(\frac{an}{qM_1^{1+t}}\right)
\nonumber\\
&&
\int_{\mathbb{R}}U\left(\frac{\zeta}{M^\varepsilon}\right)
  \omega(q M_1^t,\zeta)
  e\left(\frac{n\zeta}{qQM_1^{1+t}}\right)\mathrm{d}\zeta
   +O\left(M^{-A}\right).
\ena
With the above procedure repeated, we further transform the above sums into a
form of summing over terms with $(q,M_1M_2)=1$. More precisely, we have
\bna
\delta(n)&=& \sum_{s=0,1}\;
\sum_{\ell=0}^{[\log Q/\log M_2]}
\frac{1}{Q}
\sum_{q\leq Q/M_2^\ell \atop (q,M)=1}
\frac{1}{q M_1M_2^\ell}\;
\sideset{}{^*}\sum_{a(\text{{\rm mod }} q M_1^{1-s} M_2^\ell)}
e\left(\frac{an}{q M_1^{1-s} M_2^\ell}\right)\qquad\qquad\\
&\times&
\int_{\mathbb{R}} U\left(\frac{\zeta}{M^\varepsilon}\right)\;
  \omega(q M_2^\ell,\zeta)\;
  e\left(\frac{n\zeta}{qQM_1M_2^\ell}\right)\mathrm{d}\zeta\\
&+&\sum_{\ell=0}^{[\log Q/\log M_2]}\;
\sum_{t=1}^{[\log Q/\log M_1]}
\frac{1}{Q}
\sum_{q\leq Q/M_1^t M_2^{\ell}\atop (q,M)=1}\frac{1}{qM_1^{1+t}M_2^{\ell}}\;
\sideset{}{^*}\sum_{a(\text{{\rm mod }} qM_1^{1+t}M_2^{\ell})}
e\left(\frac{an}{qM_1^{1+t}M_2^{\ell}}\right)\qquad\qquad\\
&\times&
\int_{\mathbb{R}}U\left(\frac{\zeta}{M^\varepsilon}\right)
  \omega(q M_1^t M_2^{\ell},\zeta)\;
  e\left(\frac{n\zeta}{q QM_1^{1+t}M_2^{\ell}}\right)
  \mathrm{d}\zeta+O\left(M^{-A}\right).
\ena
This proves the lemma.
\end{proof}

Write $\mathcal{S}_r(N)$ in \eqref{initial sum} as
\bna
\mathcal{S}_r(N)=\sum_n\lambda_{\pi}(n,r)W\left(\frac{n}{N}\right)
\sum_{m}\lambda_g(m)\chi(m)V\left(\frac{m}{N}\right)\delta\left(m-n\right).
\ena
Then by apply Lemma \ref{circle method} with $Q=\sqrt{N/M_1}$, we get
\bna
\mathcal{S}_r(N)
&=&\sum_n\lambda_{\pi}(n,r)W\left(\frac{n}{N}\right)
\sum_{m}\lambda_g(m)\chi(m)V\left(\frac{m}{N}\right)\\
&\times& \bigg\{
\sum_{s=0,1}\;
\sum_{\ell=0}^{[\log Q/\log M_2]}
\frac{1}{Q}
\sum_{q\leq Q/M_2^\ell \atop (q,M)=1}
\frac{1}{q M_1M_2^\ell}\;
\sideset{}{^*}\sum_{a(\text{{\rm mod }} q M_1^{1-s} M_2^\ell)}
e\left(\frac{a (m-n)}{q M_1^{1-s} M_2^\ell}\right)
\\
&&\int_{\mathbb{R}} U\left(\frac{\zeta}{M^\varepsilon}\right)\;
  \omega(q M_2^\ell,\zeta)\;
  e\left(\frac{(m-n) \zeta}{qQM_1M_2^\ell}\right)\mathrm{d}\zeta
\\
&&\left.+\sum_{\ell=0}^{[\log Q/\log M_2]}\;
\sum_{t=1}^{[\log Q/\log M_1]}
\frac{1}{Q}
\sum_{q\leq Q/M_1^t M_2^{\ell}\atop (q,M)=1}\frac{1}{qM_1^{1+t}M_2^{\ell}}\;
\sideset{}{^*}\sum_{a(\text{{\rm mod }} qM_1^{1+t}M_2^{\ell})}
e\left(\frac{a (m-n)}{qM_1^{1+t}M_2^{\ell}}\right) \right.\\
&& \int_{\mathbb{R}}U\left(\frac{\zeta}{M^\varepsilon}\right)
  \omega(q M_1^t M_2^{\ell},\zeta)\;
  e\left(\frac{(m-n)\zeta}{q QM_1^{1+t}M_2^{\ell}}\right)
  \mathrm{d}\zeta
 \bigg\}  +O_A\left(M^{-A}\right),
\ena
where $W(x)\in C_c^{\infty}(1/2,5/2)$ satisfying
$W(x)=1$ for $x\in [1,2]$ and  $W^{(j)}(x)\ll_j 1$ for any integer $j\geq 0$.

In the above sum, we only consider the first term in braces with $s=0$ and $\ell=0$, that is,
\bna
\mathcal{S}_r^\flat(N)
&:=&\sum_n\lambda_{\pi}(n,r)W\left(\frac{n}{N}\right)
\sum_{m}\lambda_g(m)\chi(m)V\left(\frac{m}{N}\right)\\
&&\times
\frac{1}{Q}
\sum_{q\leq Q \atop (q,M)=1}
\frac{1}{q M_1}\;
\sideset{}{^*}\sum_{a(\text{{\rm mod }} q M_1 )}
e\left(\frac{a (m-n)}{q M_1 }\right)
\int_{\mathbb{R}} U\left(\frac{\zeta}{M^\varepsilon}\right)\;
  \omega(q,\zeta)\;
  e\left(\frac{(m-n) \zeta}{qQM_1}\right)\mathrm{d}\zeta.
 \ena
The other terms are lower order terms, which can be treated similarly.
We arrange $\mathcal{S}_r^\flat(N)$ as
\bna
\mathcal{S}_r^\flat(N)
&=&
\frac{1}{QM_1}\int_{\mathbb{R}}
U\left(\frac{\zeta}{M^\varepsilon}\right)
  \sum_{q\leq Q \atop (q,M)=1}\frac{\omega(q,\zeta)}{q}\;
  \sideset{}{^*}\sum_{a(\text{{\rm mod }} qM_1 )}
  \sum_m\lambda_g(m)\chi(m)e\left(\frac{am}{qM_1}\right)
  V\left(\frac{m}{N}\right)
  \\&&
  \times
  e\left(\frac{m\zeta}{qQM_1}\right)
  \sum_n\lambda_{\pi}(n,r)e\left(-\frac{an}{qM_1}\right)
  W\left(\frac{n}{N}\right)e\left(\frac{-n\zeta}{qQM_1}\right)\mathrm{d}\zeta\\
&:=&
\frac{1}{QM_1}\int_{\mathbb{R}}
U\left(\frac{\zeta}{M^\varepsilon }\right)
  \sum_{q\leq Q \atop (q,M)=1}\frac{\omega(q ,\zeta)}{q}
  \sideset{}{^*}\sum_{a(\text{{\rm mod }} qM_1 )}
  \mathscr{A}\times \mathscr{B} \hspace{3pt} \mathrm{d}\zeta,
\ena
where
$$
\mathscr{A}=\sum_m\lambda_g(m)\chi(m)
e\left(\frac{am}{qM_1}\right)
 V\left(\frac{m}{N}\right)e\left(\frac{m\zeta}{qQM_1}\right),
$$
and
$$
\mathscr{B}=\sum_n\lambda_{\pi}(n,r)
e\left(-\frac{an}{qM_1}\right)
 W\left(\frac{n}{N}\right)e\left(\frac{-n\zeta}{qQM_1}\right).
$$

\subsection{Voronoi formulas}
Next we transform $\mathscr{A}$ and $\mathscr{B}$ by $\rm GL_2$ and $\rm GL_3$
Voronoi formulas, respectively, and obtain the following results.

\begin{lemma}\label{GL2 lemma}
$\mathscr{A}$ is approximately $\mathscr{A}_1+\mathscr{A}_2+O_A(M^{-A})$
for any $A>0$, where
\bna
\mathscr{A}_1=\frac{ N^{1/2}}{\tau(\overline{\chi})}\sum_\pm
\sum_{c(\text{{\rm mod }} M)\atop c\not\equiv -\overline{q}aM_2\bmod M_1}
\overline{\chi}(c)
\sum_{m\leq M^{1+\varepsilon} M_2}\frac{\lambda_g(m)}{m^{1/2}}
e\left(\pm\frac{\overline{aM_2+cq}}{qM }m\right)
\mathfrak{I}^\pm\left(\frac{mN}{q^2 M^2 },q,\zeta\right),
\ena
and
\bna
\mathscr{A}_2=
\frac{N^{1/2}}{\tau(\overline{\chi})}
\sum_\pm
\sum_{c(\bmod M)\atop c\equiv-\overline{q}aM_2 \bmod M_1}\overline{\chi}(c)
\sum_{m\leq M_2^{2+\epsilon}/M_1}
\frac{\lambda_g(m)}{m^{1/2}}
e\left(\pm\frac{\overline{(aM_2+cq)/M_1}}{qM_2}m\right)
\mathfrak{I}^\pm\left(\frac{mN}{q^2M_2^2}, q,\zeta\right).
\ena
Here $\tau(\chi)$ is the Gauss sum, for $x\ll M^{\varepsilon}$,
\bna
\mathfrak{I}^\pm\left(x,q,\zeta\right)=x^{1/2}\int_0^{\infty}V(y)e\left(\frac{\zeta N y}{qQM_1 }\right)
\mathbf{J}_g^{\pm}(4\pi\sqrt{xy})\mathrm{d}y,
\ena
and for $x\gg M^{\varepsilon}$,
\bna
\mathfrak{I}^\pm\left(x,q,\zeta\right)=x^{1/4} \int_0^\infty y^{-1/4}V(y)
e\left(\frac{\zeta N y}{qQM_1 }\pm 2 \sqrt{xy}\right)\mathrm{d}y.
\ena
\end{lemma}

\begin{lemma}\label{GL3 lemma}
We have
\bna
\mathscr{B}&=&\frac{N^{1/2}}{q^{1/2}M_1^{1/2} r^{1/2}}\sum_{\pm}\sum_{n_{1}|qM_1  r}\;
\sum_{n_1^2n_2\leq \frac{N^{2+\varepsilon}r}{Q^3}}
\frac{\lambda_{\pi}(n_1,n_2)}{n_2^{1/2}}
  S\left(-r\overline{a},\pm n_{2};\frac{qM_1 r}{n_{1}}\right)\nonumber\\&&
  \qquad\qquad\qquad\qquad\qquad\qquad\qquad
\times\mathfrak{J}^{\pm}\left(\frac{n_{1}^{2}n_{2}}{q^3M_1^3 r},q,\zeta\right)+O(M^{-A})
\ena
for any $A>0$, where
\bna
\mathfrak{J}^{\pm}
\left(x,q,\zeta\right)=\frac{1}{2\pi}
\int_{\mathbb{R}}(Nx)^{-i\tau}
\gamma_{\pm}\left(-\frac{1}{2}+i\tau\right) W^{\dag}\left(\frac{\zeta N}{qQM_1},\frac{1}{2}-i\tau\right)\mathrm{d}\tau.
\ena
\end{lemma}

\bigskip

By Lemmas \ref{GL2 lemma} and \ref{GL3 lemma} we have
\bea\label{initial decomposition}
\mathcal{S}_r^{\flat}(N)=\mathbf{S}_1+\mathbf{S}_2+O_A(M^{-A}),
\eea
where $\mathbf{S}_j=\mathbf{S}_j(r,N)$, $j=1,2$, are defined by
\bea\label{S_j}
\mathbf{S}_j&=&\frac{N}{r^{1/2}Q M_1^{3/2}\tau(\overline{\chi})}\sum_{\pm}\sum_{\pm}
\mathop{\sum\sum}_{n_1^2n_2\leq \frac{N^{2+\varepsilon}r}{Q^3}}
  \frac{\lambda_{\pi}(n_1,n_2)}{n_2^{1/2}}
   \sum_{q\leq Q, (q,M)=1\atop n_1|qM_1r}
    \frac{1}{q^{3/2}}\nonumber\\
 &\times& \sum_{m\leq N_j^\ast}\frac{\lambda_g(m)}{m^{1/2}}\,
 \mathfrak{C}_j(n_{1},\pm n_{2}, \pm m, q)\,
 \mathfrak{R}^{\pm,\pm}\left( Y_j, \frac{n_{1}^2 n_{2}}{ q^3 M_1^3r},q\right)
\eea
with
\bea\label{N_j}
N_1^\ast= M^{1+\varepsilon}M_2,\qquad N_2^\ast= M_2^{2+\varepsilon}/M_1,\qquad
Y_1=\frac{mN}{q^2 M^2},\qquad Y_2=\frac{mN}{q^2 M_2^2},
\eea
\bea\label{C_1}
\mathfrak{C}_1(n_{1}, n_{2},  m, q)= \sideset{}{^*} \sum_{ a \bmod qM_1}
\sum_{c(\text{{\rm mod }} M)\atop c\not\equiv -\overline{q}aM_2\bmod M_1}\overline{\chi}(c)
S\left(-r\overline{a}, n_{2};\frac{qM_1 r}{n_{1}}\right)
  e\left(\frac{\overline{aM_2+cq}}{qM}m\right),
  \eea
  \bea\label{C_2}
\mathfrak{C}_2(n_{1}, n_{2},  m, q)=\sideset{}{^*} \sum_{ a \bmod qM_1}
\sum_{c(\text{{\rm mod }} M)\atop c\equiv-\overline{q}aM_2 \bmod M_1}
\overline{\chi}(c)
S\left(-r\overline{a}, n_{2};\frac{qM_1 r}{n_{1}}\right)
  e\left(\frac{\overline{(aM_2+cq)/M_1}}{qM_2}m\right)
\eea
and
\bea\label{R integral}
\mathfrak{R}^{\pm,\pm}(y_1,y_2,q)=
\int_{\mathbb{R}}U\left(\frac{\zeta}{M^\varepsilon}\right)
  \omega(q,\zeta)\mathfrak{I}^\pm\left(y_1,q,\zeta\right)
  \mathfrak{J}^{\pm}\left(y_2,q,\zeta\right)
  \mathrm{d}\zeta.
\eea
Reducing the $n_1,n_2$-sum, $q$-sum and $m$-sum into dyadic intervals, we have
\bea\label{S_j after reducing}
\mathbf{S}_j=\sum_{\pm}\sum_{\pm} \sum_{ L\ll\frac{N^{2+\varepsilon}r}{Q^3} \atop L \,\mathrm{dyadic}}
\sum_{N_j \ll N_j^ \ast \atop N_j \,\mathrm{dyadic}}
\sum_{ C\ll Q \atop C \,\mathrm{dyadic}} \mathbf{S}_j\left( C,L,N_j,\pm,\pm\right),
\eea
where
\bna
 \mathbf{S}_j\left( C,L,N_j,\pm,\pm\right)=
 \frac{N}{r^{1/2}Q M_1^{3/2} \tau(\overline{\chi}) } \mathop{\sum\sum}_{n_1^2n_2\sim L}
 \frac{\lambda_{\pi}(n_1,n_2)}{n_2^{1/2}}
\sum_{ q\sim C, (q,M)=1 \atop n_{1}|qM_1  r } \frac{1}{q^{3/2}}&&\nonumber \\
  \qquad\qquad\qquad\qquad\qquad\qquad\qquad
 \times \sum_{m\sim N_j}\frac{\lambda_g(m)}{m^{1/2}}\,
 \mathfrak{C}_j(n_{1},\pm n_{2}, \pm m, q)\,
 \mathfrak{R}^{\pm,\pm}\left( Y_j, \frac{n_{1}^2 n_{2}}{ q^3 M_1^3r},q\right).
\ena

We only consider the case $ (n_1,M_1)=1$,
since the contribution from the case $ M_1 | n_1$ is smaller. Under this assumption and the assumption in
\eqref{asummption on r},
the condition $n_1| q M_1 r$ is reduced to $n_1|qr$ and by abusing notations we still denote
the contribution from terms with $ (n_1,M_1)=1$ by
$ \mathbf{S}_j\left( C,L,N_j,\pm,\pm\right)$, i.e.,
\bea \label{S_j_CL}
 \mathbf{S}_j\left( C,L,N_j,\pm,\pm\right)=
 \frac{N}{r^{1/2}Q M_1^{3/2} \tau(\overline{\chi}) } \mathop{\sum\sum}_{n_1^2n_2\sim L}
 \frac{\lambda_{\pi}(n_1,n_2)}{n_2^{1/2}}
\sum_{ q\sim C, (q,M)=1 \atop n_{1}|qr } \frac{1}{q^{3/2}}&&\nonumber \\
  \qquad\qquad\qquad\qquad\qquad\qquad\qquad
 \times \sum_{m\sim N_j}\frac{\lambda_g(m)}{m^{1/2}}\,
 \mathfrak{C}_j(n_{1},\pm n_{2}, \pm m, q)\,
 \mathfrak{R}^{\pm,\pm}\left( Y_j, \frac{n_{1}^2 n_{2}}{ q^3 M_1^3r},q\right).
\eea

Before further analysis, we make a computation of the character sums
$ \mathfrak{C}_j(n_{1}, n_{2}, m, q)$, $j=1,2$. Set
\bea\label{Kl_2}
\mathrm{K}l_2(n; M_1)=\frac{1}{M_1^{1/2}}\sum_{x_1 x_2\equiv n (\bmod M_1)}\;
e\left(\frac{x_1+x_2}{M_1}\right)
\eea
and
\bea\label{L}
\mathrm{L}_{\alpha,\beta}(v; M_1)= \frac{1}{M_1^{1/2}} \sum_{b(\text{{\rm mod }} M_1)
\atop (b+\beta v , M_1)=1} \overline{\chi_1}(b)
e\left(\frac{\alpha \overline{b+\beta v}}{M_1}\right).
\eea
We have the following result.

\begin{lemma}\label{initial character sum}
We have
\bea\label{C1}
\mathfrak{C}_1(n_{1}, n_{2},  m, q)=
\chi_1(q) \chi_2(q^2 M_1 \overline{m})\tau( \chi_2) M_1
\mathfrak{B}(n_1,n_2,m;q)\mathfrak{D}(n_1,n_2,m,q;M_1)
\eea
and
\bea\label{C2}
\mathfrak{C}_2(n_{1}, n_{2},  m, q)=
\chi_1( q\overline{M_2 r} (qr/n_1)^2 )
\chi_2(q^2  \overline{mM_1})
\tau^2(\chi_1)\tau(\chi_2)
\mathfrak{B}(n_1,n_2,M_1^2m;q),
\eea
where
\bea\label{B-def}
\mathfrak{B}(n_1,n_2,m;q)=\sum_{d|q} d \mu\left(\frac{q}{d}\right)
\sideset{}{^*}\sum_{ u  \bmod qr/n_1 \atop { m \equiv M_2^2n_1 u \bmod d }}
e\left(\frac{n_2 \overline{M_1u}}{qr/n_1}\right)
\eea
and
\bea\label{D-def}
\mathfrak{D}(n_1,n_2,m,q;M_1)=\sideset{}{^*}\sum_{a(\text{{\rm mod }} M_1)}\;
\mathrm{L}_{m\overline{qM_2}, M_2}(a; M_1)
\mathrm{K}l_2(-r n_2 \overline{a (qr/n_1)^2}; M_1).
\eea

\end{lemma}
\begin{proof}
By \eqref{C_1}, we compute
\bna
\mathfrak{C}_1(n_{1}, n_{2},  m, q)&=& \sideset{}{^*} \sum_{ a \bmod qM_1}
\sum_{c(\text{{\rm mod }} M)\atop c\not\equiv -\overline{q}aM_2\bmod M_1}  \overline{\chi}(c)
S\left(-r\overline{aM_1}, n_{2}\overline{M_1};\frac{q r}{n_{1}}\right)\\
&&S\left(-r\overline{a} \overline{qr/n_1}, n_{2}\overline{qr/n_1}; M_1\right)
e\left(\frac{m\overline{\left(a M_2+ cq\right)q}}{M}\right)
e\left(\frac{m \overline{a M M_2}}{q}\right)\\
&=&
\sideset{}{^*}\sum_{a_1 \bmod q}\;S\left(-r\overline{a_1 M_1}, n_{2}\overline{M_1};q r/n_1\right)
e\left(\frac{m \overline{a_1 M_1 M_2^2}}{q}\right)
\sideset{}{^*} \sum_{ c_2 \bmod M_2} \overline{\chi_2}(c_2)
e\left(\frac{m \overline{c_2 q^2 M_1}}{M_2}\right)\\
&\times& \sideset{}{^*} \sum_{a_2 \bmod M_1}
\sideset{}{^*}\sum_{c_1(\text{{\rm mod }} M_1)
\atop c_1\not\equiv -\overline{q}a_2M_2\bmod M_1}  \overline{\chi_1}(c_1)
S\left(-r\overline{a_2} \overline{qr/n_1}, n_{2}\overline{qr/n_1}; M_1\right)
e\left(\frac{m\overline{\left(a_2 M_2+ c_1 q\right)q M_2}}{M_1}\right),
\ena
where the first sum over $ a_1 \bmod q$ equals
\bna
&&\sideset{}{^*} \sum_{u \bmod qr/n_1}
e\left(\frac{n_2 \overline{M_1u}}{qr/n_1}\right)
\sideset{}{^*}\sum_{a_1 \bmod q}e\left(\frac{m\overline{M_1M_2^2}-n_1\overline{M_1}u}{q}\overline{a_1}\right)\\
&=& \sum_{d|q} d \mu\left(\frac{q}{d}\right)
\sideset{}{^*}\sum_{ u  \bmod qr/n_1 \atop { m \equiv M_2^2n_1 u \bmod d }}
e\left(\frac{n_2 \overline{M_1u}}{qr/n_1}\right),
\ena
the second sum over $c_2 \bmod M_2$ is $ \chi_2(q^2 M_1 \overline{m})\tau(\chi_2)$ and
 and the first sum over $ a_1 \bmod q$ equals
\bna
&&\sideset{}{^*} \sum_{u \bmod qr/n_1} e\left(\frac{n_2 \overline{M_1}\overline{u}}{qr/n_1}\right)
\sideset{}{^*}\sum_{a_1 \bmod q}e\left(\frac{m\overline{MM_2}-n_1\overline{M_1}u}{q}\overline{a_1}\right)\\
&=& \sum_{d|q} d \mu\left(\frac{q}{d}\right) \sideset{}{^*}
\sum_{u  \bmod qr/n_1 \atop { m \equiv M_2^2n_1 u \bmod d }}
e\left(\frac{n_2 \overline{M_1}\overline{u}}{qr/n_1}\right).
\ena

Finally, the last double sum over $ a_2 \bmod M_1$ and $ c_1 \bmod M_1$ is equal to
\bna
&&\mathop{\sideset{}{^*}\sum_{a_2 (\text{{\rm mod }} M_1)}\;
\sideset{}{^*}\sum_{c_1(\text{{\rm mod }} M_1)}}_{(a_2 M_2+c_1 q, M_1)=1 }
\overline{\chi_1}(c_1)
S\left(-r\overline{a_2} \overline{qr/n_1}, n_{2}\overline{qr/n_1}; M_1\right)
e\left(\frac{m\overline{\left(a_2 M_2+ c_1 q\right)q M_2}}{M_1}\right)\\
&=& \chi_1(q)M_1^{\frac{1}{2}}\mathop{\sideset{}{^*}\sum_{a (\text{{\rm mod }} M_1)}\;
\sideset{}{^*}\sum_{b(\text{{\rm mod }} M_1)}}_{(b+aM_2, M_1)=1 }
\overline{\chi_1}(b)
e\left(\frac{m\overline{\left(b+aM_2\right)q M_2}}{M_1}\right)
\mathrm{K}l_2(-r n_2 \overline{a (qr/n_1)^2}; M_1)
\\
&=&
\chi_1(q) M_1\;
\sideset{}{^*}\sum_{a(\text{{\rm mod }} M_1)}\;
\mathrm{L}_{m\overline{qM_2}, M_2}(a; M_1)
\mathrm{K}l_2(-r n_2 \overline{a (qr/n_1)^2}; M_1)
\ena
with $\mathrm{K}l_2(n; M_1)$ and $\mathrm{L}_{\alpha,\beta}(v; M_1)$ defined in \eqref{Kl_2}
and \eqref{L}, respectively. Here we have used the following identity for the Ramanujan sum
\bna
R_q(b)= \sideset{}{^*}\sum_{a \bmod q}\; e\left( \frac{ba}{q} \right)=\sum_{d|(q,b)} d \mu\left(\frac{q}{d}\right).
\ena
Assembling the above computations, we obtain
\bna
\mathfrak{C}_1(n_{1}, n_{2},  m, q)=
\chi_1(q) \chi_2(q^2 M_1 \overline{m})\tau( \chi_2) M_1\mathfrak{B}(n_1,n_2,m;q)
\mathfrak{D}(n_1,n_2,m,q;M_1),
\ena
where $\mathfrak{B}(n_1,n_2,m;q)$ and $\mathfrak{D}(n_1,n_2,m,q;M_1)$ are defined in \eqref{B-def}
and \eqref{D-def}, respectively.

Similarly, by \eqref{C_2},
\bna
\mathfrak{C}_2(n_{1}, n_{2},  m, q)
&=& \sideset{}{^*} \sum_{ a \bmod qM_1}
\sum_{c(\text{{\rm mod }} M)\atop c\equiv -\overline{q}aM_2\bmod M_1}  \overline{\chi}(c)
S\left(-r\overline{aM_1}, n_{2}\overline{M_1};qr/n_1\right)S\left(-r\overline{a} \overline{qr/n_1}, n_{2}\overline{qr/n_1}; M_1\right)\\
&&
e\left(\frac{q\overline{\left(a M_2+ cq\right)/M_1}}{M_2}m\right)
e\left(\frac{M_2 \overline{\left(a M_2+ cq\right)/M_1}}{q}m\right)\\
&=&
\sideset{}{^*}\sum_{a_1 \bmod q}\;S\left(-r\overline{a_1 M_1}, n_{2}\overline{M_1};qr/n_1\right)
e\left(\frac{m M_1 \overline{a_1 M_2^2}}{q}\right)
\sideset{}{^*} \sum_{ c_2 \bmod M_2} \overline{\chi_2}(c_2)
e\left(\frac{m M_1 \overline{c_2 q^2}}{M_2}\right)\\
&\times& \sideset{}{^*} \sum_{a_2 \bmod M_1}
\overline{ \chi_1}( -\overline{q}a_2M_2)
S\left(-r\overline{a_2} \overline{qr/n_1}, n_{2}\overline{qr/n_1}; M_1\right),
\ena
where the first sum over $ a_1 \bmod q$ is
\bna
&&\sideset{}{^*} \sum_{u \bmod qr/n_1}
e\left(\frac{n_2 \overline{M_1u}}{qr/n_1}\right)
\sideset{}{^*}\sum_{a_1 \bmod q}
e\left(\frac{m M_1\overline{M_2^2}-n_1\overline{M_1}u}{q}\overline{a_1}\right)\\
&=& \sum_{d|q} d \mu\left(\frac{q}{d}\right)
\sideset{}{^*}\sum_{ u  \bmod qr/n_1
\atop { m \equiv \overline{M_1^2} M_2^2 n_1u \bmod d }}
e\left(\frac{n_2 \overline{M_1u}}{qr/n_1}\right)
=\mathfrak{B}(n_1,n_2,M_1^2m;q),
\ena
where $\mathfrak{B}(n_1,n_2,m;q)$ is defined in \eqref{B-def},
the second sum over $c_2 \bmod M_2$ equals $ \chi_2(q^2  \overline{mM_1})\tau(\chi_2)$ and the last sum over $ a_2 \bmod M_1$ is equal to
\bna
&&\chi_1(-q\overline{M_2})
\sideset{}{^*} \sum_{v \bmod M_1}
e\left(\frac{n_2 \overline{qr/n_1}\overline{v}}{M_1}\right)
\sideset{}{^*} \sum_{a_2 \bmod M_1}
\chi_1(\overline{a_2})
e\left(\frac{- r \overline{a_2} \overline{qr/n_1}v}{M_1}\right)\\
&=&
\chi_1(-q\overline{M_2}) \sideset{}{^*} \sum_{v \bmod M_1}
e\left(\frac{n_2 \overline{qr/n_1}\overline{v}}{M_1}\right)
\chi_1( -\overline{ v r}qr/n_1 )\tau(\chi_1)\\
&=&
\chi_1(-q\overline{M_2})
\chi_1( -\overline{ n_2 r}(qr/n_1)^2 )\tau^2(\chi_1)\\
&=&
\chi_1( q\overline{M_2 r} (qr/n_1)^2 )\tau^2(\chi_1).
\ena
Therefore
\bna
\mathfrak{C}_2(n_{1}, n_{2},  m, q)=
\chi_1( q\overline{M_2 r} (qr/n_1)^2 )
\chi_2(q^2  \overline{mM_1})
\tau^2(\chi_1)\tau(\chi_2)
\mathfrak{B}(n_1,n_2,M_1^2m;q).
\ena
This proves the lemma.
\end{proof}

\subsection{Cauchy-Schwarz and Poisson summation}

We return to the analysis of $\mathbf{S}_j\left( C,L,N_j,\pm,\pm\right)$ in
\eqref{S_j_CL}.
Applying the Cauchy-Schwarz inequality to $n_1, n_2$-sum in \eqref{S_j_CL}
and using the Rankin-Selberg estimate in \eqref{GL3 RS},one has that
$\mathbf{S}_j\left( C,L,N_j,\pm,\pm\right)$ is bounded by
\bna
 \frac{N L^{1/2}}{r^{1/2}Q M_1^{3/2} M^{1/2} }
 \bigg( \mathop{\sum\sum}_{n_1^2n_2\sim L} \frac{1}{n_2}
 \bigg| \sum_{ q\sim C, (q,M)=1 \atop n_{1}|qr } q^{-3/2}
 \sum_{m\sim N_j}\frac{\lambda_g(m)}{m^{1/2}}\,
 \mathfrak{C}_j(n_{1},\pm n_{2}, \pm m, q)\,
 \mathfrak{R}^{\pm,\pm}\left( Y_j, \frac{n_{1}^2 n_{2}}{ q^3 M_1^3r},q\right)\bigg|^2\bigg)^{1/2}.
\ena
As in Munshi \cite{Munshi2018}, we write $q=q_1q_2$ with $q_1|(n_1 r)^{\infty} $, $(q_2, n_1 r)=1 $ and apply Cauchy-Schwarz again to get that the expression inside the absolute value being
\bna
&&\sum_{ n_{1}|q_1r \atop q_1|(n_1r)^{\infty}}q_1^{-3/2}
\sum_{ q_2\sim C/q_1 \atop (q_2, q_1 M)=1}q_2^{-3/2}
 \sum_{m\sim N_j}\frac{\lambda_g(m)}{m^{1/2}}\,
 \mathfrak{C}_j(n_{1},\pm n_{2}, \pm m, q_1q_2)\,
 \mathfrak{R}^{\pm,\pm}\left( Y_j, \frac{n_{1}^2 n_{2}}{ q_1^3 q_2^3 M_1^3r},q_1q_2\right)\\
 & \ll&
 N^\varepsilon \bigg( \sum_{ n_{1}|q_1r \atop q_1|(n_1r)^{\infty}}q_1^{-3}
 \bigg|\sum_{ q_2\sim C/q_1 \atop (q_2, q_1 M)=1}q_2^{-3/2}
  \sum_{m\sim N_j}\frac{\lambda_g(m)}{m^{1/2}}\,
 \mathfrak{C}_j(n_{1},\pm n_{2}, \pm m, q_1q_2)\,
 \mathfrak{R}^{\pm,\pm}\left( Y_j, \frac{n_{1}^2 n_{2}}{ q_1^3 q_2^3 M_1^3r},q_1q_2\right)
 \bigg|^2 \bigg)^{1/2}.
\ena
Hence we have
\bea\label{Sj upper bound}
\mathbf{S}_j\left( C,L,N_j,\pm,\pm\right) \ll \frac{N L^{1/2}}{r^{1/2}Q M_1^{3/2} M^{1/2} }
\bigg\{ \sum_{n_1\leq Cr}\sum_{n_1|q_1r \atop q_1|(n_1r)^{\infty}}q_1^{-3}
\times \mathbf{\Omega}_j^{\pm}( n_1,q_1,r)\bigg\}^{1/2},
\eea
where
\bea\label{Omega_j}
\mathbf{\Omega}_j^{\pm}( n_1,q_1,r)&=&
\sum_{n_2}\frac{1}{n_2}\phi\left( \frac{n_2}{L/n_1^2}\right)
\bigg| \sum_{ q_2\sim C/q_1 \atop (q_2, q_1 M)=1}q_2^{-3/2}
  \sum_{m\sim N_j}\frac{\lambda_g(m)}{m^{1/2}}\nonumber\\&&\times
 \mathfrak{C}_j(n_{1},\pm n_{2}, \pm m, q_1q_2)\,
 \mathfrak{R}^{\pm,\pm}\left( Y_j, \frac{n_{1}^2 n_{2}}{ q_1^3 q_2^3 M_1^3r},q_1q_2\right) \bigg|^2.
\eea
Here $\phi$ is a nonnegative smooth function on $\mathbb{R}^+$, supported on $[2/3,3]$, and such that $ \phi(x)=1$ for $x\in[1,2]$.

\subsubsection{Estimation of $ \mathbf{\Omega}_1^{\pm}( n_1,q_1,r)$ }
We first estimate the contribution from $ \mathbf{\Omega}_1^{\pm}( n_1,q_1,r)$.
Note that
\bna
n_2\mapsto \mathfrak{C}_1(n_{1},\pm n_{2}, \pm m, q_1q_2)\,
\overline{\mathfrak{C}_1(n_{1},\pm n_{2}, \pm m', q_1q_2')}
\ena
is periodic of period $M_1q_2q_2'q_1r/n_1$. Opening the absolute square in \eqref{Omega_j},
we break the $n_2$-sum into congruence classes modulo $M_1q_2q_2'q_1r/n_1 $ and then
apply the Poisson summation formula to the sum over $n_2$. It is therefore sufficient
to consider the following sum:
\bea\label{Omega_1}
\widetilde{\mathbf{\Omega}}_1^{\pm}( n_1,q_1,r)&=&
 \sum_{ q_2\sim C/q_1 \atop (q_2, q_1 M)=1}q_2^{-3/2}
 \sum_{ q_2'\sim C/q_1 \atop (q_2', q_1 M)=1}q_2'^{-3/2}
  \sum_{m\sim N_1}\frac{1}{m^{1/2}}\,
  \sum_{m'\sim N_1}\frac{|\lambda_g(m')|^2}{m'^{1/2}}\nonumber\\
&\times&
\sum_{\widetilde{n}_2\in\mathbf{Z}}\left|\mathcal{C}(\widetilde{n}_2)\right|\,
\bigg| \mathcal{H}^{\pm,\pm}\left( \frac{\widetilde{n}_2 L}{M_1q_2q_2'q_1n_1r}\right)\bigg|,
\eea
where the character sum
$ \mathcal{C}(\widetilde{n}_2)= \mathcal{C}(\widetilde{n}_2,m,m',q_2,q_2',n_1,q_1,r)$ is given by
\bea\label{C}
\mathcal{C}(\widetilde{n}_2)=\frac{n_1}{M_1q_2q_2'q_1r}\sum_{ v \bmod M_1q_2q_2'q_1r/n_1}
\mathfrak{C}_1(n_{1},\pm v, \pm m, q_1q_2)\,
\overline{\mathfrak{C}_1(n_{1},\pm v, \pm m^{'}, q_1q_2')}
e\left(\frac{\widetilde{n}_2v}{M_1q_2q_2'q_1r/n_1}\right),
\eea
and the integral $\mathcal{H}^{\pm,\pm}(X) = \mathcal{H}^{\pm,\pm}(X;m,m',q_2,q_2',q_1,r)$ is given by
\bea\label{H}
\mathcal{H}^{\pm,\pm}(X)=\int_{\mathbb{R}}\phi(\xi )\,
\mathfrak{R}^{\pm,\pm}\left( \frac{mN}{q_1^2q_2^2 M^2}, \frac{L\xi }{ q_1^3 q_2^3 M_1^3r},q_1q_2\right)
\overline{\mathfrak{R}^{\pm,\pm}\left( \frac{m'N}{q_1^2q_2'^2 M^2}, \frac{L\xi}{ q_1^3 q_2'^3 M_1^3r},q_1q_2'\right)}
e( -X\xi )\frac{\mathrm{d}\xi}{\xi}.
\eea
Here we have used the estimate $|\lambda_g(m)\lambda_g(m')|\leq  |\lambda_g(m)|^2+|\lambda_g(m')|^2$ and
we recall that $ Y_1=mN/(q^2 M^2)$.

\subsubsection{Estimation of the integral $\mathcal{H}^{\pm,\pm}(X)$}

\begin{lemma}\label{integral estimate}
The integral $\mathcal{H}^{\pm,\pm}(X)$ is negligibly small unless
$|X|\ll M^{\varepsilon}Q/C$, in which case
\bna
\mathcal{H}^{\pm,\pm}(X)\ll M^{\varepsilon}  C/Q.
\ena
\end{lemma}

\begin{proof}
By \eqref{R integral},
\bna
\xi^\ell \frac{\partial^\ell }{\partial \xi^\ell}\mathfrak{R}^{\pm,\pm}\left(Y_j,\frac{L\xi}{q^3M_1^3r},q\right)=
\int_{\mathbb{R}}U\left(\frac{\zeta}{M^\varepsilon}\right)
  \omega(q,\zeta)\mathfrak{I}^\pm\left(Y_j,q,\zeta\right)
  \xi^\ell \frac{\partial^\ell }{\partial \xi^\ell}\mathfrak{J}^{\pm}\left(\frac{L\xi}{q^3M_1^3r},q,\zeta\right)
  \mathrm{d}\zeta,
\ena
where by \eqref{J-integral definition},
\bna
\xi^\ell \frac{\partial^\ell }{\partial \xi^\ell}\mathfrak{J}^{\pm}\left(\frac{L\xi}{q^3M_1^3r},q,\zeta\right)
&=&
\frac{1}{2\pi}
\int_{\mathbb{R}}(-i\tau)(-i\tau-1)\cdots (-i\tau-\ell+1) \\&&
\left(\frac{NL\xi}{q^3M_1^3r} \right)^{-i\tau}\gamma_{\pm}\left(-\frac{1}{2}+i\tau\right)
W^{\dag}\left(\frac{\zeta N}{qQM_1},\frac{1}{2}-i\tau\right)\mathrm{d}\tau.
\ena
By definition,
\bna
W^{\dag}\left(\frac{\zeta N}{qQM_1},\frac{1}{2}-i\tau\right)=
\int_0^{\infty}W(v_2)v_2^{-1/2}e\left( -\frac{\tau}{2\pi}\log v_2-\frac{\zeta N v_2}{qQM_1} \right)
\mathrm{d}v_2.
\ena
By repeated integration by parts, one sees that the above integral
is negligibly small unless $ |\tau|\asymp |\zeta|N/(qQM_1):=\Xi$.
Moreover, by the second derivative test for exponential integrals,
\bea\label{111}
W^{\dag}\left(\frac{\zeta N}{qQM_1},\frac{1}{2}-i\tau\right)\ll (1+|\tau|)^{-1/2}.
\eea
Thus
\bna
\xi^\ell \frac{\partial^\ell }{\partial \xi^\ell}\mathfrak{J}^{\pm}\left(\frac{L\xi}{q^3M_1^3r},q,\zeta\right)&=&
\frac{1}{2\pi}
\int_{\mathbb{R}}\varpi\Big(\frac{|\tau|}{\Xi}\Big)
(-i\tau)(-i\tau-1)\cdots (-i\tau-\ell+1) \\&&
\left(\frac{NL\xi}{q^3M_1^3r} \right)^{-i\tau}\gamma_{\pm}\left(-\frac{1}{2}+i\tau\right)
W^{\dag}\left(\frac{\zeta N}{qQM_1},\frac{1}{2}-i\tau\right)\mathrm{d}\tau+O(M^{-A}),
\ena
where $\varpi(x)\in \mathcal{C}_c^{\infty}(0,\infty)$
satisfying $\varpi^{(j)}(x)\ll_j 1$ for any integer $j\geq 0$.
Hence
\bea\label{3}
\xi^\ell \frac{\partial^\ell }{\partial \xi^\ell}\mathfrak{R}^{\pm,\pm}\left(Y_j,\frac{L\xi}{q^3M_1^3r},q\right)&=&\frac{1}{2\pi}
\int_{\mathbb{R}}U\left(\frac{\zeta}{M^\varepsilon}\right)
  \omega(q,\zeta)\int_{\mathbb{R}}\varpi\Big(\frac{|\tau|}{\Xi}\Big)
  (-i\tau)(-i\tau-1)\cdots (-i\tau-\ell+1)\nonumber\\&&
  \left(
\int_0^{\infty}W(v_2)v_2^{-1/2}e\left( -\frac{\tau}{2\pi}\log v_2-\frac{\zeta N v_2}{qQM_1} \right)
\mathrm{d}v_2\right)\nonumber
\\&&    \left(\frac{NL\xi}{q^3M_1^3r} \right)^{-i\tau}
\gamma_{\pm}\left(-\frac{1}{2}+i\tau\right)
\mathfrak{I}^\pm\left(Y_j,q,\zeta\right)
\mathrm{d}\tau
  \mathrm{d}\zeta+O(M^{-A}),
\eea
where by Lemma \ref{GL2 lemma} for $Y_j\gg M^{\varepsilon}$,
\bea\label{1}
\mathfrak{I}^\pm\left(Y_j,q,\zeta\right)=Y_j^{1/4} \int_0^\infty
v_1^{-1/4}V(v_1)e\left(\frac{\zeta N v_1}{qQM_1 }\pm 2 \sqrt{Y_jv_1}\right)\mathrm{d}v_1,
\eea
and for $Y_j\ll M^{\varepsilon}$,
\bea\label{2}
\mathfrak{I}^\pm\left(Y_j,q,\zeta\right)=Y_j^{1/2}\int_0^{\infty}V(v_1)e\left(\frac{\zeta N v_1}{qQM_1 }\right)
\mathbf{J}_g^{\pm}(4\pi\sqrt{Y_j v_1})\mathrm{d}v_1.
\eea

Consider the $\zeta$-integral
\bna
\int_{\mathbb{R}}U\left(\frac{\zeta}{M^\varepsilon}\right)
  \omega(q,\zeta)e\left(\frac{\zeta N (v_1-v_2)}{qQM_1} \right)
    \mathrm{d}\zeta.
\ena
Repeated integration by parts shows that the above integral
is negligibly small unless $ |v_1-v_2|\ll qQM_1^{1+\varepsilon}/N\ll CM^{\varepsilon}/Q$ for $q\sim C$
and $Q=\sqrt{N/M_1}$, where we have used \eqref{rapid decay omega}. Moreover,
by Stirling's formula (see \cite{Munshi2}),
\bna
\gamma_{\pm}\left(-\frac{1}{2}+i\tau\right)=\left(\frac{|\tau|}{e\pi}\right)^{3i\tau}\Upsilon_{\pm}(\tau),
\qquad \Upsilon_{\pm}^{(j)}(\tau)\ll |\tau|^{-j}
\ena
for any integer $j\geq 0$.
Thus for $Y_j\gg M^{\varepsilon}$,
we write
\bna
\xi^\ell \frac{\partial^\ell }{\partial \xi^\ell}
\mathfrak{R}^{\pm,\pm}\left(Y_j,\frac{L\xi}{q^3M_1^3r},q\right)&=&\frac{Y_j^{1/4}}{2\pi}
\int_{\mathbb{R}}U\left(\frac{\zeta}{M^\varepsilon}\right)
  \omega(q,\zeta)
\int_{|v|\ll CM^{\varepsilon}/Q}e\left(\frac{\zeta N v}{qQM_1} \right)\\&&
 \times \int_{\mathbb{R}}\int_0^\infty V_0(v_1,\tau;v)
e\left( f(v_1,\tau;v)\right)
\mathrm{d}v_1\mathrm{d}\tau\mathrm{d}v
  \mathrm{d}\zeta.
\ena
where
\bna
V_0(v_1,\tau;v)=v_1^{-1/4}(v_1+v)^{-1/2}V(v_1)W(v_1+v)
(-i\tau)(-i\tau-1)\cdots (-i\tau-\ell+1) \Upsilon_{\pm}(\tau)
\varpi\Big(\frac{|\tau|}{\Xi}\Big)
\ena
and
\bna
f(v_1,\tau;v)=-\frac{\tau}{2\pi}\log\frac{NL\xi}{q^3M_1^3r}-\frac{3\tau}{2\pi}\log\frac{|\tau|}{e\pi}
-\frac{\tau}{2\pi}\log (v_1+v)\pm 2 \sqrt{Y_jv_1}.
\ena
Note that
\bna
\frac{\partial f(v_1,\tau;v)}{\partial v_1}
     &=&-\frac{\tau}{2\pi (v_1+v)}\pm Y_j^{1/2}v_1^{-1/2},\\
\frac{\partial f(v_1,\tau;v)}{\partial \tau}
     &=&-\frac{1}{2\pi}\log\frac{NL\xi}{q^3M_1^3r}-\frac{3}{2\pi}\log\frac{|\tau|}{\pi}
     -\frac{1}{2\pi}\log (v_1+v).
\ena
Thus we have
\bna
\frac{\partial^2 f(v_1,\tau;v)}{\partial v_1^2}
     &=&\frac{\tau}{2\pi (v_1+v)^2}\mp \frac{1}{2}Y_j^{1/2}v_1^{-3/2}\gg |\tau|,\\
\frac{\partial^2 f(v_1,\tau;v)}{\partial \tau^2}
     &=&-\frac{3}{2\pi \tau}\gg |\tau|^{-1},\\
\frac{\partial^2 f(v_1,\tau;v)}{\partial v_1\partial \tau}&=& -\frac{1}{2\pi (v_1+v)}.
\ena
This implies that
\bna
\left|\mathrm{det}f''\right|=\left|\frac{\partial^2 f}{\partial v_1^2}
\frac{\partial^2 f}{\partial \tau^2}
   -\left(\frac{\partial^2 f}{\partial v_1\partial \tau}\right)^2\right|
   \gg 1
\ena

By applying the two dimensional second derivative test in
\cite[Lemma 5.1.3]{Hux2} (see also \cite[Lemma 4]{Munshi2}) with
$\rho_1=|\tau|, \rho_2=|\tau|^{-1}$,
we have
\bna
\int_{\mathbb{R}}\int_0^\infty V_0(v_1,\tau;v)
e\left( f(v_1,\tau;v)\right)
\mathrm{d}v_1\mathrm{d}\tau\mathrm{d}v
  \mathrm{d}\zeta\ll \Xi^{\ell}\ll  \left(\frac{|\zeta|N}{qQM_1}\right)^{\ell}
  \ll \left(\frac{Q}{q}\right)^{\ell}
\ena
recalling $ Q= \sqrt{N/M_1}$. It follows that (note that in this case $\Xi\asymp \sqrt{Y_j}$ by \eqref{1})
\bna
\xi^\ell \frac{\partial^\ell }{\partial \xi^\ell}
\mathfrak{R}^{\pm,\pm}\left(Y_j,\frac{L\xi}{q^3M_1^3r},q\right)\ll
M^\varepsilon \left(\frac{Q}{C}\right)^{\ell-1/2}  \qquad \mbox{for\,}\;  q\sim C.
\ena
This holds also for $Y_j\ll M^{\varepsilon}$. In fact, by \eqref{3} and \eqref{2} we have
\bna
\xi^\ell \frac{\partial^\ell }{\partial \xi^\ell}\mathfrak{R}^{\pm,\pm}
\left(Y_j,\frac{L\xi}{q^3M_1^3r},q\right)&=&\frac{Y_j^{1/2}}{2\pi}
\int_0^{\infty}V(v_1)\mathbf{J}_g^{\pm}(4\pi\sqrt{Y_j v_1})
\nonumber\\&&\int_{\mathbb{R}}\varpi\Big(\frac{|\tau|}{\Xi}\Big)
  (-i\tau)(-i\tau-1)\cdots (-i\tau-\ell+1)
\left(\frac{NL\xi}{q^3M_1^3r} \right)^{-i\tau}
\gamma_{\pm}\left(-\frac{1}{2}+i\tau\right)\nonumber\\&&
\int_{\mathbb{R}}U\left(\frac{\zeta}{M^\varepsilon}\right)
  \omega(q,\zeta)e\left(\frac{\zeta N v_1}{qQM_1 }\right)W^{\dag}\left(\frac{\zeta N}{qQM_1},\frac{1}{2}-i\tau\right)
\mathrm{d}v_1
\mathrm{d}\tau
  \mathrm{d}\zeta+O(M^{-A}).
\ena
As before, by considering the $\zeta$-integral and integrating by parts, one sees that
 the above integral
is negligibly small unless $ |v_1-v_2|\ll qQM_1^{1+\varepsilon}/N\ll CM^{\varepsilon}/Q$.
Furthermore, by \eqref{Bessel-small} we have
$\mathbf{J}_g^{\pm}(4\pi\sqrt{Y_j v_1})\ll Y_j^{-1/4}$. Combining these estimates with
\eqref{111}, we obtain
\bna
\xi^\ell \frac{\partial^\ell }{\partial \xi^\ell}\mathfrak{R}^{\pm,\pm}
\left(Y_j,\frac{L\xi}{q^3M_1^3r},q\right)\ll Y_j^{1/4}\frac{CM^{\varepsilon}}{Q}
\Xi^{\ell+1/2}\ll M^{\varepsilon}\left(\frac{Q}{C}\right)^{\ell-1/2}.
\ena

Hence by applying integration by parts repeatedly on the $\xi$-integral in \eqref{H} and evaluating the
resulting $\xi$-integral trivially, we find that
\bna
\mathcal{H}^{\pm,\pm}(X)\ll_\ell M^\varepsilon\; \frac{C}{Q}\;\left(\frac{XC}{Q}\right)^{-\ell}
\ena
for any integer $ \ell \geq 0$. Therefore,
$ \mathcal{H}^{\pm,\pm}(X)$ is negligibly small unless $ X \leq N^\varepsilon Q/C$.
Moreover, by taking $ \ell=0$, one has
\bna
\mathcal{H}^{\pm,\pm}(X)\ll_\ell M^\varepsilon C/Q.
\ena

\end{proof}

By Lemma \ref{integral estimate}, the sum over $ \widetilde{n}_2$ in \eqref{Omega_1} can be restricted to, up to a negligibly error,
\bea\label{widetilde{n}_2}
\widetilde{n}_2 \ll Q^{1+\varepsilon}M_1Cn_1r/(q_1L) := L^*.
\eea

\subsubsection{Estimation of the character sum $\mathcal{C}(\widetilde{n}_2)$}
In this section we estimate the character sums in \eqref{C}.
By \eqref{C1}, we have
\bea\label{character sums decom}
\mathcal{C}(\widetilde{n}_2)&=&
\frac{n_1}{M_1q_2q_2'q_1r}
\sum_{ v_1 \bmod M_1}\;
\sum_{ v_2 \bmod q_2q_2'q_1r/n_1}
\chi_1(q_1q_2) \chi_2(\pm q_1^2q_2^2 M_1 \overline{m})\tau( \chi_2) M_1
\mathfrak{B}(n_1,\pm v_2,\pm m;q_1q_2)\nonumber\\
&\times&
\mathfrak{D}(n_1,\pm v_1,\pm m,q_1q_2;M_1)
\overline{\chi}_1(q_1q_2')\overline{\chi}_2( \pm q_1^2q_2^2M_1\overline{m'})
\overline{\tau(\chi_2)}M_1
\overline{\mathfrak{B}(n_1,\pm v_2,\pm m';q_1q_2')}\nonumber\\
&\times&
\overline{\mathfrak{D}(n_1,\pm v_1,\pm m',q_1q_2';M_1)}
e\left(\frac{\overline{q_2q_2'q_1r/n_1}v_1\widetilde{n}_2}{M_1}\right)
e\left(\frac{\overline{M_1}v_2\widetilde{n}_2}{q_2q_2'q_1r/n_1}\right)\nonumber\\
&\leq& M_1^2M_2|\mathcal{C}_1(\widetilde{n}_2)\mathcal{C}_2(\widetilde{n}_2)|,
\eea
where
\bea\label{C11}
\mathcal{C}_1(\widetilde{n}_2)=
\frac{1}{M_1}\sum_{ v \bmod M_1}
\mathfrak{D}(n_1,\pm v,\pm m,q_1q_2;M_1)
\overline{\mathfrak{D}(n_1,\pm v,\pm m',q_1q_2';M_1)}
e\left(\frac{\overline{q_2q_2'q_1r/n_1}\widetilde{n}_2 v}{M_1}\right)
\eea
and
\bea\label{C22}
\mathcal{C}_2(\widetilde{n}_2)=
\frac{n_1}{q_2q_2'q_1r}
\sum_{ v \bmod q_2q_2'q_1r/n_1}
\mathfrak{B}(n_1,\pm v,\pm m;q_1q_2)
\overline{\mathfrak{B}(n_1,\pm v,\pm m';q_1q_2')}
e\left(\frac{\overline{M_1}\widetilde{n}_2 v }{q_2q_2'q_1r/n_1}\right).
\eea

We first evaluate $\mathcal{C}_1(\widetilde{n}_2)$ in \eqref{C11}.
Following closely Lin, Michel and Sawin \cite{LMS}, we transform $\mathcal{C}_1(\widetilde{n}_2)$ to made it
to a sheaf-theoretic treatment.
Define $ \mathbf{K}(v):= \mathrm{K}l_2( \gamma v;M_1)$,
$\mathbf{L}(v)= \mathrm{L}_{\alpha, \beta }(v; M_1)$ and their normalized multiplicative convolution
\bna
\mathbf{K} \star \mathbf{L}(v)= \frac{1}{M_1^{1/2}}
\sum_{u\in \mathbf{F}_{M_1}^\times }
\mathbf{K}(u) \mathbf{L}(v/u)
=
\frac{1}{M_1^{1/2}}
\sum_{u \in \mathbf{F}_{M_1}^\times }
\mathbf{K}(vu) \mathbf{L}(1/u).
\ena
We also define likewise $\mathbf{K}'(v) ,\mathbf{L}'(v)$ with parameters $\alpha', \beta', \gamma'$.
In particular, in our case we shall choose the parameters as the following:
\bea\label{parameters}
&&\alpha=\pm m\overline{q_1q_2}\overline{M_2},\quad \beta=M_2,\quad
\gamma=\mp r\overline{( q_1q_2r/n_1)^2}, \quad
\eta=\widetilde{n}_2\overline{q_2q_2'q_1r/n_1},\nonumber\\
&&\alpha'=\pm m'\overline{q_1q_2'}\overline{M_2}, \quad \beta'=M_2,\quad
\gamma'=\mp r\overline{(q_1q_2'r/n_1)^2}.
\eea
Then by the definition of $\mathfrak{D}(n_1,n_2,m,q;M_1)$ in \eqref{D-def}, one has
\bea\label{transform}
\mathcal{C}_1(\widetilde{n}_2)&=&
\frac{1}{M_1}\sum_{ v \bmod M_1}\;
\sideset{}{^*} \sum_{a_1 \bmod M_1}
\mathrm{L}_{\pm m\overline{q_1q_2M_2}, M_2}(a_1; M_1)
\mathrm{K}l_2(\mp r v\overline{a_1 (q_1q_2r/n_1)^2}; M_1)\nonumber\\
&\times&
\sideset{}{^*} \sum_{a_2 \bmod M_1}\;
\overline{\mathrm{L}_{\pm m'\overline{q_1q_2'M_2}, M_2}(a_2; M_1)}
\mathrm{K}l_2(\mp r v\overline{a_2 (q_1q_2'r/n_1)^2}; M_1)
e\left(\frac{\overline{q_2q_2'q_1r/n_1}\widetilde{n}_2v}{M_1}\right)\nonumber\\
&=& \sum_{v\bmod M_1}\;\mathbf{K} \star \mathbf{L}(v)\,
\overline{\mathbf{K'} \star \mathbf{L'}(v)}e\left(\frac{\eta v}{M_1}\right).
\eea
Define the Fourier transform of the function
$K: (\mathbf{Z}/M_1\mathbf{Z})^{\times}\rightarrow \mathbb{C}$ by
\bna
\widehat{K}(v)=\frac{1}{M_1^{1/2}}\sum_{a\in \mathbf{F}_{M_1}}K(a)e\left(\frac{av}{M_1}\right).
\ena
Then by Plancherel formula we have (see \cite[Section 4]{LMS})
\bea\label{fourier}
\sum_{v \bmod M_1} \mathbf{K} \star\mathbf{L} (v)
\overline{\mathbf{K'} \star \mathbf{L'} (v)}
e\left(\frac{\eta v}{M_1}\right)
=\sum_{v}
\widehat{\mathbf{K} \star \mathbf{L}} (v)\,
\overline{\widehat{\mathbf{K'} \star \mathbf{L'}} (v-\eta)}
:=\sum_{v} \mathbf{Z}(v) \overline{\mathbf{Z'}(v-\eta)},
\eea
where
\bea\label{Z}
\mathbf{Z}(v)&=&\widehat{\mathbf{K} \star \mathbf{L}} (v)=
\frac{1}{M_1^{1/2}}\sum_{a\in \mathbf{F}_{M_1}}\mathbf{K} \star \mathbf{L}(a)
e\left(\frac{av}{M_1}\right)\nonumber\\
&=&
\frac{1}{M_1^{1/2}}
\sum_{a\in \mathbf{F}_{M_1}} \frac{1}{M_1^{1/2}}
\sum_{u\in \mathbf{F}_{M_1}^\times }
\mathbf{K}(u) \mathbf{L}(a/u)e\left(\frac{a v}{M_1}\right)\nonumber\\
&=&
\frac{1}{M_1^{1/2}}
\sum_{u\in \mathbf{F}_{M_1}^\times}\mathbf{K}(u)
\frac{1}{M_1^{1/2}}
\sum_{a\in \mathbf{F}_{M_1}} \mathbf{L}(a/u) e\left(\frac{av}{M_1}\right)\nonumber\\
&=&
\frac{1}{M_1^{1/2}}
\sum_{u\in \mathbf{F}_{M_1}^\times}\mathbf{K}(u)
\widehat{\mathbf{L}}(u v).
\eea
By the definition of $\mathbf{L}(v):=\mathrm{L}_{\alpha, \beta}(v; M_1) $ in \eqref{L} and the Fourier expansion of $\chi_1$ in terms of additive characters (see (3.12) in \cite{IK})), we have
\bna
\mathbf{L}(v)
&=& \frac{1}{M_1^{1/2}} \sum_{b \bmod M_1 \atop (b+\beta v , M_1)=1} \frac{1}{\tau(\chi_1)}
\sum_{c \bmod M_1} \chi_1(c) e\left( \frac{bc}{M_1}\right)
 e\left(\frac{\alpha \overline{b+\beta v}}{M_1}\right)\\
 &=&
 \frac{1}{\tau(\chi_1) M_1^{1/2}}\sum_{c \bmod M_1} \chi_1(c)
 \sum_{b \bmod M\atop (b+\beta v , M_1)=1}
 e\left(\frac{\alpha \overline{b+\beta v}+c\left(b+\beta a\right)}{M_1}\right)
 e\left( - \frac{\beta cv}{M_1}\right)\\
 &=&
  \frac{1}{\tau(\chi_1)}\sum_{c \bmod M_1} \chi_1(c)
  \mathrm{K}l_2( \alpha c ;M_1)e\left( - \frac{\beta cv}{M_1}\right),
\ena
where $\mathrm{K}l_2( n;M_1)$ is defined in \eqref{Kl_2}. It follows that
\bna
\widehat{\mathbf{L}}(v)=
\frac{M_1^{1/2}}{\tau(\chi_1)}
\chi_1(\beta^{-1}v)
 \mathrm{K}l_2(\beta^{-1}\alpha  v; M_1)
\ena
and by \eqref{Z},
\bna
\mathbf{Z}(v)=\frac{1}{\tau(\chi_1)}
\sum_{u\in \mathbf{F}_{M_1}^\times} \mathrm{K}l_2 (\beta\gamma u; M_1)
\chi_1(uv)\mathrm{K}l_2(\alpha u v; M_1).
\ena

We quote the following results of Lin, Michel and Sawin \cite[Proposition 4.5]{LMS}.
\begin{proposition}\label{prop}
Let $T_{\mathcal{F}}(\mathbf{F}_{M_1})$ be the subgroup of $\mathbf{F}_{M_1}^\times$ defined by
\bna
T_{\mathcal{F}}(\mathbf{F}_{M_1}^\times)
=\left\{\lambda\in \mathbf{F}_{M_1}^\times,
 [\times \lambda]^* \mathcal{F}\;
\text{{is geometrically isomorphic to }} \mathcal{F}\right\}
\ena
Assume that the sheaf $\mathcal{F}$ is good. For
any $\alpha, \beta, \alpha', \beta', \gamma, \gamma',\eta \in \mathbf{F}_{M_1}^\times $, we have
\bna
\sum_{v} \mathbf{Z}(v) \overline{\mathbf{Z}'(v-\eta)}= O(M_1^{1/2}).
\ena
If $\eta=0$ the above bound holds unless
\bna
\alpha/\alpha'=\beta\gamma/(\beta'\gamma') \in T_{\mathcal{F}}(\mathbf{F}_{M_1})
\ena
in which case
\bna
\sum_{v} \mathbf{Z}(v) \overline{\mathbf{Z}'(v)}=
c_{\mathcal{F}}(\alpha/\alpha') M_1+O(M_1^{1/2})
\ena
for $c_{\mathcal{F}}(\alpha/\alpha')$ some complex number of modulus 1. Here the implicit constants depend only on $C(\mathcal{F})$.
\end{proposition}

By \eqref{parameters}-\eqref{fourier} and Proposition \ref{prop}, we have
\bea\label{C2 estimate}
\mathcal{C}_1(\widetilde{n}_2)\ll M_1^{1/2}
\eea
unless
\bea\label{condition}
m q_2'/(m'q_2)=q_2'^2/q_2^2 \in T_{\mathcal{F}}(\mathbf{F}_{M_1})
\eea
in which case
\bea\label{condition result}
\mathcal{C}_1(\widetilde{n}_2)=
c_{\mathcal{F}}\left(m q_2'/(m'q_2)\right) M_1+O\left(M_1^{1/2}\right),\qquad
|c_{\mathcal{F}}\left(m q_2'/(m'q_2)\right)|=1.
\eea

\medskip

Next, we evaluate $\mathcal{C}_2(\widetilde{n}_2)$ in \eqref{C2}.
By \eqref{B-def} and \eqref{C2}, we have
\bna
\mathcal{C}_2(\widetilde{n}_2)&=&
\frac{n_1}{q_2q_2'q_1r}
\sum_{ v \bmod q_2q_2'q_1r/n_1}\;
\sum_{d|q_1q_2} d \mu\left(\frac{q_1q_2}{d}\right)
\sideset{}{^*}\sum_{u \bmod q_1q_2r/n_1 \atop n_1 u\equiv \pm \overline{M_2^2}m \bmod d}
e\left(\frac{\pm v \overline{M_1 u}}{q_1q_2r/n_1}\right)\;\nonumber
\\
&\times&
\sum_{d'|q_1q_2'} d' \mu\left(\frac{q_1q_2'}{d'}\right)
\sideset{}{^*}\sum_{u' \bmod q_1q_2'r/n_1 \atop n_1 u'\equiv \pm \overline{M_2^2}m' \bmod d'}
e\left(\frac{\mp v \overline{M_1u'}}{q_1q_2'r/n_1}\right)\;
e\left(\frac{\overline{M_1}\widetilde{n}_2 v}{q_2 q_2'q_1 r/n_1}\right)\nonumber
\\
&=&
\sum_{d|q_1q_2} d \mu\left(\frac{q_1q_2}{d}\right)
\sum_{d'|q_1q_2'} d' \mu\left(\frac{q_1q_2'}{d'}\right)
\mathop{\sideset{}{^*}\sum_{u \bmod q_1q_2r/n_1 \atop n_1 u\equiv \pm \overline{M_2^2}m \bmod d}\;
\sideset{}{^*}\sum_{u' \bmod q_1q_2'r/n_1 \atop n_1 u'\equiv \pm \overline{M_2^2}m' \bmod d'}\;}_{q_2\overline{u'}-q_2'\overline{u}  \equiv \pm \widetilde{n}_2 \bmod q_2q_2'q_1r/n_1}
1.
\ena
This expression for $\mathcal{C}_2(\widetilde{n}_2)$ has been estimated in Lin and Sun
By \cite[Lemma 4.2]{LS} (see (5.1) in \cite{LS}). More precisely, we have the following result.
\begin{lemma}\label{mathcal{C}_2 estimate}
We have
\bea\label{nonzero n2}
\mathcal{C}_2(\widetilde{n}_2)\ll
\mathop{\sum\sum}_{d_1, d_1'|q_1} d_1 d_1'
 \mathop{\sideset{}{^*}\sum_{u \bmod q_1r/n_1 \atop n_1u\equiv \pm \overline{M_2^2}m\bmod d_1}\;
\sideset{}{^*}\sum_{u' \bmod q_1r/n_1 \atop
n_1u'\equiv \pm \overline{M_2^2}m'\bmod d_1'}}_{
q_2\overline{u'} -q_2'\overline{u}\equiv \pm \widetilde{n}_2\bmod q_1r/n_1}
\mathop{\sum\sum}_{d_2|(q_2, M_2^2 q_2'n_1\pm m \widetilde{n}_2)
\atop d_2'|(q_2', M_2^2 q_2n_1\mp m' \widetilde{n}_2)} d_2 d_2'.
\eea
Moreover, if $\widetilde{n}_2=0$, we have
\bea\label{q2q2-0}
q_2=q_2'
\eea
and
\bea\label{C(0)}
\mathcal{C}_2(0)\ll  q_1q_2r\mathop{\sum\sum}_{d, d'|q_1q_2 \atop (d, d')|(m-m')}(d, d').
\eea
\end{lemma}

\section{Contribution from the zero frequency}

In this section we evaluate the contribution to $ \mathbf{\widetilde{\Omega}}_1^{\pm}( n_1,q_1,r)$
in \eqref{Omega_1}
(and in turn to $ \mathbf{S}_1\left( C,L,N_1,\pm,\pm\right)$ by \eqref{Sj upper bound})
from the terms with $ \widetilde{n}_2 =0$.

Denote by $\mathbf{\Delta}_0^{\pm}$ ( resp. $\mathbf{\Sigma}_0^{\pm}$ ) the contribution
from the terms with $ \widetilde{n}_2 =0$
to $\widetilde{\mathbf{\Omega}}_1^{\pm}( n_1,q_1,r)$ ( resp. $ \mathbf{S}_1\left( C,L,N_1,\pm,\pm\right)$ ).
By \eqref{q2q2-0} and \eqref{parameters}, we have
\bna
q_2=q_2',\qquad \gamma=\gamma', \qquad\eta= \widetilde{n}_2\overline{ q_2^2q_1r/n_1}.
\ena
By \eqref{C2 estimate}-\eqref{condition result}, one has $\widetilde{ \mathcal{C}}_1(0) \ll M_1^{1/2}$ unless $ m/m'=1 \in T_{\mathcal{F}}(\mathbf{F}_{M_1})$ in which case
\bna
\widetilde{ \mathcal{C}}_1(0)=c_{\mathcal{F}}(1) M_1+O\left(M_1^{1/2}\right),\qquad
|c_{\mathcal{F}}(1)|=1.
\ena
Then by \eqref{Omega_1}, \eqref{character sums decom}, \eqref{C(0)}
and Lemma \ref{integral estimate}, we have
\bea\label{compare 1}
\mathbf{\Delta}_0^{\pm}
&\ll&  \sum_{ q_2\sim C/q_1 \atop (q_2, q_1 M)=1}\frac{1}{q_2^3}
  \sum_{m\sim N_1}\frac{1}{m^{1/2}}\,
  \sum_{m'\sim N_1}\frac{|\lambda_g(m')|^2}{m'^{1/2}}M_1^2M_2
\left|\mathcal{C}_1(0)\right|\,\left|\mathcal{C}_2(0)\right|\,
\left| \mathcal{H}^{\pm,\pm}(0)\right|\nonumber\\
&\ll&  \frac{M^{1+\varepsilon}M_1C}{QN_1}\sum_{ q_2\sim C/q_1 \atop (q_2, q_1 M)=1}q_2^{-3}
\mathop{\sum\sum}_{m,m'\sim N_1}|\lambda_g(m')|^2
\left(M_1 \delta_{m/m'=1\in T_{\mathcal{F}}(\mathbf{F}_{M_1})} +M_1^{1/2} \right)q_1q_2r
\mathop{\sum\sum}_{d, d'|q_1q_2 \atop (d, d')|(m-m')}(d, d')\nonumber\\
&\ll&  \frac{M^{1+\varepsilon}M_1C^2r}{QN_1}\sum_{ q_2\sim C/q_1 \atop (q_2, q_1 M)=1}q_2^{-3}\sum_{ \ell |q_1q_2}\ell
\mathop{\sum\sum}_{m,m'\sim N_1 \atop \ell|(m-m')}|\lambda_g(m')|^2
\left(M_1 \delta_{m/m'=1\in T_{\mathcal{F}}(\mathbf{F}_{M_1})} +M_1^{1/2} \right)\nonumber\\
&\ll&  \frac{M^{1+\varepsilon}M_1C^2r}{Q}\sum_{ q_2\sim C/q_1 \atop (q_2, q_1 M)=1}q_2^{-3}\sum_{ \ell |q_1q_2}\ell
\left(M_1\left(\frac{N_1}{M_1 \ell}+1\right)+M_1^{1/2}\left(\frac{N_1}{\ell}+1\right) \right)\nonumber\\
&\ll&  \frac{M^{1+\varepsilon}M_1C^2r}{Q}\left(\frac{q_1}{C}\right)^2
\left(M_1\left(\frac{N_1}{M_1 }+C\right)+M_1^{1/2}(N_1+C) \right)\nonumber\\
&\ll&  \frac{M^{1+\varepsilon}M_1q_1^2r}{Q}
\left(M_1C+M_1^{1/2}N_1 \right)\\
&\ll&  \frac{M^{1+\varepsilon}M_1q_1^2r}{Q}
\left(M_1Q+ \frac{M^2}{M_1^{1/2}} \right)\nonumber
\eea
recalling \eqref{N_j} that $ N_1\leq N_1^*=M^{1+\varepsilon}M_2$ and $C\ll Q$.
Here we have used Rankin-Selberg's estimate in \eqref{GL2 RS}.

Plugging this bound into \eqref{Sj upper bound} one has
\bna
\mathbf{\Sigma}_0^\pm
&\ll&
\frac{N^{1+\varepsilon} L^{1/2}}{r^{1/2}Q M_1^{3/2} M^{1/2} }
\bigg\{ \sum_{n_1 \leq Cr}\sum_{n_1|q_1r \atop q_1|(n_1r)^{\infty}}q_1^{-3}
\frac{M M_1q_1^2r}{Q}
\bigg(M_1Q+ \frac{M^2}{M_1^{1/2}} \bigg)\bigg\}^{1/2}\\
&\ll&
\frac{N^{1+\varepsilon} L^{1/2}}{r^{1/2}Q M_1^{3/2} M^{1/2} }
\frac{M^{1/2} M_1^{1/2} r^{1/2}}{Q^{1/2}}
\bigg(M_1^{1/2}Q^{1/2}+ MM_1^{-1/4} \bigg)\\
&\ll&
\frac{N^{1+\varepsilon} L^{1/2}}{M_1 Q^{3/2}}
\bigg(M_1^{1/2}Q^{1/2}+ MM_1^{-1/4} \bigg)
\ena
Therefore, the contribution from the terms with $ \widetilde{n}_2 =0$ to $\mathbf{S}_1$ in
\eqref{S_j after reducing} is at most
\bea\label{S1-zero}
&&\sum_{ L\ll\frac{N^{2+\varepsilon}r}{Q^3} \atop L \,\mathrm{dyadic}}
\sum_{N_1 \ll N_1^ \ast \atop N_1 \,\mathrm{dyadic}}
\sum_{ C\ll Q \atop C \,\mathrm{dyadic}} \frac{N^{1+\varepsilon} L^{1/2}}{M_1 Q^{3/2}}
\bigg(M_1^{1/2}Q^{1/2}+ MM_1^{-1/4} \bigg)\nonumber\\
&\ll&\frac{N^{1+\varepsilon}}{M_1 Q^{3/2}}\left(\frac{N^2r}{Q^3}\right)^{1/2}
\bigg(M_1^{1/2}Q^{1/2}+MM_1^{-1/4} \bigg)\nonumber\\
&\ll&r^{1/2}N^{1/2+\varepsilon}M_1^{1/2}
\bigg(N^{1/4}M_1^{1/4}+ MM_1^{-1/4} \bigg)
\eea
recalling $Q=\sqrt{N/M_1}$.

\section{Contribution from the nonzero frequencies}

In this section we evaluate the contribution to $ \mathbf{\widetilde{\Omega}}_1^{\pm}( n_1,q_1,r)$
in \eqref{Omega_1}
(and in turn to $ \mathbf{S}_1\left( C,L,N_1,\pm,\pm\right)$ by \eqref{Sj upper bound})
from the terms with $ \widetilde{n}_2 \neq 0$.

Denote by $\mathbf{\Delta}_{\neq 0}^{\pm}$ ( resp. $\mathbf{\Sigma}_{\neq 0}^{\pm}$ ) the contribution
from the terms with $ \widetilde{n}_2 \neq 0$
to $\widetilde{\mathbf{\Omega}}_1^{\pm}( n_1,q_1,r)$ ( resp. $ \mathbf{S}_1\left( C,L,N_1,\pm,\pm\right)$ ).
We distinguish two cases according as $\widetilde{n}_2\not\equiv 0 \bmod M_1$ or
$\widetilde{n}_2\equiv 0 \bmod M_1$.
\subsection{ $\widetilde{n}_2\not\equiv 0 \bmod M_1$}
Denote the contribution from $\widetilde{n}_2\not\equiv 0 \bmod M_1$ by $\mathbf{\Delta}_{\neq 0}^{\pm,a}$.
By \eqref{Omega_1},\eqref{widetilde{n}_2} and \eqref{C2 estimate}, we have
\bna
\mathbf{\Delta}_{\neq 0}^{\pm,a}&\ll& \frac{M^{1+\varepsilon}M_1}{N_1}
 \sum_{ q_2\sim C/q_1 \atop (q_2, q_1 M)=1}q_2^{-3/2}
 \sum_{ q_2'\sim C/q_1 \atop (q_2', q_1 M)=1}q_2'^{-3/2}
  \sum_{m\sim N_1}
  \sum_{m'\sim N_1}|\lambda_g(m')|^2\nonumber\\
&\times&
\sum_{0\neq \widetilde{n}_2\ll L^*}\left|\mathcal{C}_1(\widetilde{n}_2)\right|
\left|\mathcal{C}_2(\widetilde{n}_2)\right|
\bigg| \mathcal{H}^{\pm,\pm}\left( \frac{\widetilde{n}_2 L}{M_1q_2q_2'q_1n_1r}\right)\bigg|,
\ena
Applying \eqref{character sums decom}, \eqref{nonzero n2} and Lemma \ref{integral estimate}, we obtain
\bea\label{nonzreo}
\mathbf{\Delta}_{\neq 0}^{\pm,a}&\ll& \frac{M^{1+\varepsilon}M_1^{3/2}C}{QN_1}
 \sum_{ q_2\sim C/q_1 \atop (q_2, q_1 M)=1}q_2^{-3/2}
 \sum_{ q_2'\sim C/q_1 \atop (q_2', q_1 M)=1}q_2'^{-3/2}
  \sum_{m\sim N_1}
  \sum_{m'\sim N_1}|\lambda_g(m')|^2\nonumber\\
&\times&
\sum_{0\neq \widetilde{n}_2\ll L^*}
\mathop{\sum\sum}_{d_1, d_1'|q_1} d_1 d_1'
 \mathop{\sideset{}{^*}\sum_{u \bmod q_1r/n_1 \atop n_1u\equiv \pm \overline{M_2^2}m\bmod d_1}\;
\sideset{}{^*}\sum_{u' \bmod q_1r/n_1 \atop
n_1u'\equiv \pm \overline{M_2^2}m'\bmod d_1'}}_{
q_2\overline{u'} -q_2'\overline{u}\equiv \pm \widetilde{n}_2\bmod q_1r/n_1}
\mathop{\sum\sum}_{d_2|(q_2, M_2^2 q_2'n_1\pm m \widetilde{n}_2)
\atop d_2'|(q_2', M_2^2 q_2n_1\mp m' \widetilde{n}_2)} d_2 d_2'.
\eea

Next we proceed as in \cite{LS}, Section 4.5.
Writing $q_2d_2$ in place of $q_2$ and $q_2'd_2'$ in place of $q_2'$,
and noting that for fixed $(u, d_2, d_2', q_2, q_2', n_2)$
the congruence condition
$\widetilde{n}_2\equiv d_2q_2\overline{u'}-d_2'q_2'\overline{u}\,(\bmod \,q_1 r/n_1)$
determines $u'$ uniquely, we infer
\bna
\mathbf{\Delta}_{\neq 0}^{\pm,a}&\ll&
\frac{M^{1+\varepsilon}M_1^{3/2}q_1^3}{C^2QN_1}\;\sideset{}{^*}\sum_{u\bmod q_1r/n_1}
\mathop{\sum\sum}_{d_1, d_1'|q_1} d_1 d_1'
\mathop{\sum}_{\substack {d_2\ll C/q_1\\(d_2,d_1)=1}}
\mathop{\sum}_{\substack {d_2'\ll C/q_1\\(d_2',d_1')=1}}d_2d_2'
\mathop{\sum\sum}_{\substack{q_2\sim C/q_1d_2\\q_2'\sim C/q_1d_2'}}
\\
&\times&
\sum_{0\neq \widetilde{n}_2\ll L^*}
  \mathop{\sum}_{\substack {m'\sim N_1\\M_2^2 q_2d_2n_1\mp m' \widetilde{n}_2\equiv 0\bmod d_2'}}
  |\lambda_g(m')|^2
  \mathop{\sum}_{\substack{m\sim N_1\\
                 n_1 u\equiv \pm \overline{M_2^2}m\bmod d_1\\
   M_2^2 q_2'd_2'n_1\pm m \widetilde{n}_2\equiv 0\bmod d_2}}1.
\ena
Notice that for fixed tuple $(n_1,u, \widetilde{n}_2)$ the congruences
$$
\left\{\begin{array}{l}
n_1u\equiv \pm \overline{M_2^2}m\bmod d_1\\
M_2^2 q_2'd_2'n_1\pm m \widetilde{n}_2\equiv 0\bmod d_2
\end{array}
\right.
$$
imply that $m$ is uniquely determined modulo
$d_1d_2/(d_2,\widetilde{n}_2)$. Therefore the number of $m$ is dominated by
$O\left((d_2,\widetilde{n}_2)\left(1+N_1/d_1d_2\right)\right)$.
We conclude that
\bna
\mathbf{\Delta}_{\neq 0}^{\pm,a}&\ll&
\frac{M^{1+\varepsilon}M_1^{3/2}q_1^4r}{C^2QN_1n_1}
\mathop{\sum\sum}_{d_1, d_1'|q_1} d_1 d_1'
\mathop{\sum}_{\substack {d_2\ll C/q_1\\(d_2,d_1)=1}}
\mathop{\sum}_{\substack {d_2'\ll C/q_1\\(d_2',d_1')=1}}d_2d_2'
\mathop{\sum\sum}_{\substack{q_2\sim C/q_1d_2\\q_2'\sim C/q_1d_2'}}
\\
&\times&
\sum_{0\neq \widetilde{n}_2\ll L^*\atop (d_2,\widetilde{n}_2)| q_2'd_2'n_1}
  (d_2,\widetilde{n}_2)\left(1+\frac{N_1}{d_1d_2}\right)
  \mathop{\sum}_{\substack {m'\sim N_1\\M_2^2 q_2d_2n_1\mp m' \widetilde{n}_2\equiv 0\bmod d_2'}}
  |\lambda_g(m')|^2.
\ena

\textit{Case 1}. If $M_2^2q_2d_2n_1 \mp m'\widetilde{n}_2\equiv 0\bmod d_2'$ but
$M_2^2q_2d_2n_1 \mp m_2\widetilde{n}_2\neq 0$, then
$d_2'$ is a factor of the integer $M_2^2q_2d_2n_1 \mp m'\widetilde{n}_2$. Therefore by
switching the order of summation, the $d_2'$-sum is bounded above
by $\tau(|M_2^2q_2d_2n_1 \mp m'\widetilde{n}_2|)=O(N^\varepsilon)$
with $\tau(n)$ being the divisor function, and the contribution from this case is
\bna
&\ll&
\frac{M^{1+\varepsilon}M_1^{3/2}q_1^3r}{CQN_1n_1}
\mathop{\sum\sum}_{d_1, d_1'|q_1} d_1 d_1'
\mathop{\sum}_{\substack {d_2\ll C/q_1\\(d_2,d_1)=1}}d_2
\mathop{\sum}_{\substack{q_2\sim C/q_1d_2}}
\sum_{0\neq \widetilde{n}_2\ll L^*}
  (d_2,\widetilde{n}_2)\left(1+\frac{N_1}{d_1d_2}\right)\\
&\times&
\mathop{\sum}_{\substack {d_2'\ll C/q_1\\(d_2',d_1')=1}}d_2'
  \mathop{\sum}_{\substack {m'\sim N_1\\M_2^2 q_2d_2n_1\mp m' \widetilde{n}_2\equiv 0\bmod d_2'}}
  |\lambda_g(m')|^2.
\\
&\ll&
\frac{M^{1+\varepsilon}M_1^{3/2}q_1^2r}{QN_1n_1}
\mathop{\sum\sum}_{d_1, d_1'|q_1} d_1 d_1'
\mathop{\sum}_{\substack {d_2\ll C/q_1\\(d_2,d_1)=1}}L^*\left(1+\frac{N_1}{d_1d_2}\right)N_1
\\
&\ll&
\frac{L^*M^{1+\varepsilon}M_1^{3/2}q_1^3r}{Qn_1}(C+N_1).
\ena

\textit{Case 2}. If $M_2^2q_2d_2n_1 \mp m'\widetilde{n}_2=0$,
then as long as $m'$ and $\widetilde{n}_2$ are fixed, the number of tuples $(q_2,d_2,n_1)$
is bounded above by the ternary divisor function $\tau_3(|m'\widetilde{n}_2|)$.
Hence such a contribution is dominated by
\bna
&\ll&
\frac{M^{1+\varepsilon}M_1^{3/2}q_1^4r}{C^2QN_1n_1}
\mathop{\sum\sum}_{d_1, d_1'|q_1} d_1 d_1'
\mathop{\sum}_{\substack {d_2'\ll C/q_1\\(d_2',d_1')=1}}d_2'
\mathop{\sum}_{\substack{q_2'\sim C/q_1d_2'}}\,
\sum_{0\neq \widetilde{n}_2\ll L^*}\\
&\times&
\sum_{m'\sim N_1}|\lambda_g(m')|^2
\mathop{\sum}_{\substack {d_2\ll C/q_1\atop (d_2,\widetilde{n}_2)| q_2'd_2'n_1}}
  (d_2,\widetilde{n}_2)\left(d_2+\frac{N_1}{d_1}\right)
\mathop{\sum}_{\substack{q_2\sim C/q_1d_2\\M_2^2 q_2d_2n_1\mp m' \widetilde{n}_2=0}}1
\\
&\ll&
\frac{M^{1+\varepsilon}M_1^{3/2}q_1^4r}{C^2QN_1n_1}
\mathop{\sum\sum}_{d_1, d_1'|q_1} d_1 d_1'
\mathop{\sum}_{\substack {d_2'\ll C/q_1\\(d_2',d_1')=1}}d_2'
\mathop{\sum}_{\substack{q_2'\sim C/q_1d_2'}} \;\sum_{\ell |q_2'd_2'n_1}\ell
\\
&\times&\sum_{0\neq \widetilde{n}_2\ll L^*/\ell}\;
 \mathop{\sum}_{m' \sim N_1}|\lambda_g(m')|^2
\mathop{\sum}_{\substack {d_2\ll C/q_1\atop d_2n_1| m'\widetilde{n}_2\ell}}\left(d_2+\frac{N_1}{d_1}\right)
\\
&\ll&
\frac{L^*M^{1+\varepsilon}M_1^{3/2}q_1^2r}{Q^2n_1}
\mathop{\sum\sum}_{d_1, d_1'|q_1} d_1 d_1'
\left(\frac{C}{q_1}+\frac{N_1}{d_1}\right)
\\
&\ll&
\frac{L^*M^{1+\varepsilon}M_1^{3/2}q_1^3r}{Qn_1}(C+N_1).
\ena
Assembling the above argument, we have
\bea\label{nonzreo1}
\mathbf{\Delta}_{\neq 0}^{\pm,a}\ll \frac{L^*M^{1+\varepsilon}M_1^{3/2}q_1^3r}{Qn_1}(C+N_1).
\eea

\subsection{$\widetilde{n}_2\equiv 0 \bmod M_1$}
Denote the contribution from $\widetilde{n}_2\equiv 0 \bmod M_1$ by $\mathbf{\Delta}_{\neq 0}^{\pm,b}$.
In this case we can write $\widetilde{n}_2=M_1\widetilde{n}_2'$ with $\widetilde{n}_2'\ll L^*/M_1$
and run exactly the same argument as in Section 4.1, except for replacing $L^*$ by $L^*/M_1$ and
replacing the estimate
\bna
\mathcal{C}_1(\widetilde{n}_2)\ll M_1^{1/2}
\ena
in \eqref{C2 estimate} with
\bna
\mathcal{C}_1(\widetilde{n}_2)\ll M_1
\ena
from \eqref{condition result}. Therefore, by comparison with \eqref{nonzreo1} one has
\bea\label{nonzero2}
\mathbf{\Delta}_{\neq 0}^{\pm,b}\ll \frac{L^*M^{1+\varepsilon}M_1q_1^3r}{Qn_1}
\left(C+N_1\right).
\eea

By  \eqref{nonzreo1} and \eqref{nonzero2}, we obtain
\bna
\mathbf{\Delta}_{\neq 0}^{\pm}\ll \frac{L^*M^{1+\varepsilon}M_1^{3/2}q_1^3r}{Qn_1}(C+N_1)
\ll \frac{M^{1+\varepsilon}M_1^{5/2}Qq_1^2r^2}{L}(Q+MM_2)
\ena
recalling \eqref{N_j} and \eqref{widetilde{n}_2} that $ N_1\leq N_1^*=M^{1+\varepsilon}M_2$,
$\widetilde{n}_2 \ll Q^{1+\varepsilon}M_1Cn_1r/(q_1L) := L^*$ and $C\ll Q$.
Plugging this bound into \eqref{Sj upper bound} one has
\bna
\mathbf{\Sigma}_0^\pm
&\ll&
\frac{N^{1+\varepsilon} L^{1/2}}{r^{1/2}Q M_1^{3/2} M^{1/2} }
\bigg\{ \sum_{n_1 \leq Cr}\sum_{n_1|q_1r \atop q_1|(n_1r)^{\infty}}q_1^{-3}
\frac{MM_1^{5/2}Qq_1^2r^2}{L}(Q+MM_2) \bigg\}^{1/2}\\
&\ll&
\frac{N^{1+\varepsilon} L^{1/2}}{r^{1/2}Q M_1^{3/2} M^{1/2} }
\frac{M^{1/2}M_1^{5/4}Q^{1/2}r}{L^{1/2}}(Q^{1/2}+M^{1/2}M_2^{1/2})\\
&\ll&
\frac{N^{1+\varepsilon} r^{1/2}}{M_1^{1/4} Q^{1/2}}
(Q^{1/2}+M^{1/2}M_2^{1/2}).
\ena
Therefore, the contribution from the terms with $ \widetilde{n}_2 \neq 0$ to $\mathbf{S}_1$ in
\eqref{S_j after reducing} is at most
\bea\label{S1-nonzero}
&&\sum_{ L\ll\frac{N^{2+\varepsilon}r}{Q^3} \atop L \,\mathrm{dyadic}}
\sum_{N_1 \ll N_1^ \ast \atop N_1 \,\mathrm{dyadic}}
\sum_{ C\ll Q \atop C \,\mathrm{dyadic}} \frac{N^{1+\varepsilon} r^{1/2}}{M_1^{1/4} Q^{1/2}}
(Q^{1/2}+M^{1/2}M_2^{1/2})\nonumber\\
&\ll&r^{1/2}N^{1+\varepsilon}M_1^{-1/4}
\bigg(1+ MM_1^{-1/4}N^{-1/4} \bigg)
\eea
recalling $Q=\sqrt{N/M_1}$.

\subsection{Conclusion}
By \eqref{S1-zero} and \eqref{S1-nonzero}, we have
\bea\label{S1111}
\mathbf{S}_1&\ll&r^{1/2}N^{1/2+\varepsilon}M_1^{1/2}
\bigg(N^{1/4}M_1^{1/4}+ MM_1^{-1/4} \bigg)+r^{1/2}N^{1+\varepsilon}M_1^{-1/4}
\bigg(1+ MM_1^{-1/4}N^{-1/4} \bigg)\nonumber\\
&\ll&r^{1/2}N^{1/2+\varepsilon}
\bigg(N^{1/4}M_1^{3/4}+ MM_1^{1/4} +N^{1/2+\varepsilon}M_1^{-1/4}+
N^{1/4}MM_1^{-1/2} \bigg).
\eea
Thus
\bna
&&\sum_{r\leq M^\theta}
\frac{1}{r}\sup_{M^{3-\theta}/r^2\leq N \leq M^{3+\epsilon}/r^2}
\frac{|\mathbf{S}_1|}{\sqrt N}\nonumber\\
&\ll&M^{\varepsilon}\sum_{r\leq M^\theta}
\frac{1}{r^{1/2}}\bigg(\frac{M^{3/4}M_1^{3/4}}{r^{1/2}}+MM_1^{1/4}
+\frac{M^{3/2}M_1^{-1/4}}{r}+\frac{M^{7/4}M_1^{-1/2}}{r^{1/2}}\bigg)\nonumber\\
&\ll&M^{3/4}M_1^{3/4}+M^{1+\theta/2}M_1^{1/4}
+M^{3/2}M_1^{-1/4}+M^{7/4}M_1^{-1/2}.
\ena
Furthermore, we will show in the next section that
$\mathbf{S}_2$
can be dominated by the upper bound in \eqref{S1111}.
Combining this bound with \eqref{initial decomposition} Proposition \ref{prop1} follows.

\section{Estimation of $\mathbf{S}_2$}
In this section, we will estimate $\mathbf{S}_2$ in \eqref{S_j} and its contribution to
the right hand side of \eqref{set up}. Since the proof is very similar as that for $\mathbf{S}_1$,
we will be brief.
Recall \eqref{C_2} and \eqref{Omega_j} that
\bea\label{Omega_2}
\mathbf{\Omega}_2^{\pm}( n_1,q_1,r)&=&
\sum_{n_2}\frac{1}{n_2}\phi\left( \frac{n_2}{L/n_1^2}\right)
\bigg| \sum_{ q_2\sim C/q_1 \atop (q_2, q_1 M)=1}q_2^{-3/2}
  \sum_{m\sim N_2}\frac{\lambda_g(m)}{m^{1/2}}\nonumber\\&&\times
 \mathfrak{C}_2(n_{1},\pm n_{2}, \pm m, q_1q_2)\,
 \mathfrak{R}^{\pm,\pm}\left( Y_2, \frac{n_{1}^2 n_{2}}{ q_1^3 q_2^3 M_1^3r},q_1q_2\right) \bigg|^2,
\eea
where $Y_2=mN/(q^2 M_2^2)$ and
\bna
\mathfrak{C}_2(n_{1}, n_{2},  m, q)=
\chi_1( q\overline{M_2 r} (qr/n_1)^2 )
\chi_2(q^2  \overline{mM_1})
\tau^2(\chi_1)\tau(\chi_2)
\mathfrak{B}(n_1,n_2,M_1^2m;q).
\ena
Note that
\bna
n_2\mapsto \mathfrak{C}_2(n_{1},\pm n_{2}, \pm m, q_1q_2)\,
\overline{\mathfrak{C}_2(n_{1},\pm n_{2}, \pm m', q_1q_2')}
\ena
is periodic of period $q_2q_2'q_1r/n_1$. Opening the absolute square in \eqref{Omega_2},
we break the $n_2$-sum into congruence classes modulo $q_2q_2'q_1r/n_1 $ and then
apply the Poisson summation formula to the sum over $n_2$. It is therefore sufficient
to consider the following sum:
\bea\label{Omega_22}
\widetilde{\mathbf{\Omega}}_2^{\pm}( n_1,q_1,r)&=&
 \sum_{ q_2\sim C/q_1 \atop (q_2, q_1 M)=1}q_2^{-3/2}
 \sum_{ q_2'\sim C/q_1 \atop (q_2', q_1 M)=1}q_2'^{-3/2}
  \sum_{m\sim N_2}\frac{1}{m^{1/2}}\,
  \sum_{m'\sim N_2}\frac{|\lambda_g(m')|^2}{m'^{1/2}}\nonumber\\
&\times&
\sum_{\widetilde{n}_2\in\mathbf{Z}}\left|\mathcal{D}(\widetilde{n}_2)\right|\,
\bigg| \mathcal{K}^{\pm,\pm}\left( \frac{\widetilde{n}_2 L}{q_2q_2'q_1n_1r}\right)\bigg|,
\eea
where the character sum
$ \mathcal{D}(\widetilde{n}_2)= \mathcal{C}(\widetilde{n}_2,m,m',q_2,q_2',n_1,q_1,r)$ is given by
\bna
\mathcal{D}(\widetilde{n}_2)=\frac{n_1}{q_2q_2'q_1r}\sum_{ v \bmod q_2q_2'q_1r/n_1}
\mathfrak{C}_2(n_{1},\pm v, \pm m, q_1q_2)\,
\overline{\mathfrak{C}_2(n_{1},\pm v, \pm m', q_1q_2')}
e\left(\frac{\widetilde{n}_2v}{q_2q_2'q_1r/n_1}\right),
\ena
and the integral $\mathcal{K}^{\pm,\pm}(X) = \mathcal{K}^{\pm,\pm}(X;m,m',q_2,q_2',q_1,r)$ is given by
\bna
\mathcal{K}^{\pm,\pm}(X)=\int_{\mathbb{R}}\phi(\xi )\,
\mathfrak{R}^{\pm,\pm}\left( \frac{mN}{q_1^2q_2^2 M_2^2}, \frac{L\xi }{ q_1^3 q_2^3 M_1^3r},q_1q_2\right)
\overline{\mathfrak{R}^{\pm,\pm}\left( \frac{m'N}{q_1^2q_2'^2 M_2^2}, \frac{L\xi}{ q_1^3 q_2'^3 M_1^3r},q_1q_2'\right)}
e( -X\xi )\frac{\mathrm{d}\xi}{\xi}.
\ena

The integral $\mathcal{K}^{\pm,\pm}(X)$ can be estimated exactly the same as $\mathcal{H}^{\pm,\pm}(X)$.
\begin{lemma}\label{integral estimate2}
The integral $\mathcal{K}^{\pm,\pm}(X)$ is negligibly small unless
$|X|\ll M^{\varepsilon}Q/C$, in which case
\bna
\mathcal{K}^{\pm,\pm}(X)\ll M^{\varepsilon}  C/Q.
\ena
\end{lemma}

By Lemma \ref{integral estimate2}, the sum over $ \widetilde{n}_2$ in \eqref{Omega_22}
can be restricted to, up to a negligibly error,
\bea\label{n2-2}
\widetilde{n}_2 \ll Q^{1+\varepsilon}Cn_1r/(q_1L) := L^{**}.
\eea

Next, we estimate $\mathcal{D}(\widetilde{n}_2)$.
By \eqref{C2}, we have
\bea\label{character sum-222}
|\mathcal{D}(\widetilde{n}_2)|&=& \frac{n_1}{q_2q_2'q_1r}
\bigg|
\sum_{ v \bmod q_2q_2'q_1r/n_1} \chi_1(q_1q_2\overline{M_2r}( q_1q_2r/n_1)^2 )
\chi_2(\pm q_1^2q_2^2 \overline{ m M_1})\tau^2(\chi_1 ) \tau(\chi_2 )
B( n_1, \pm v, \pm M_1^2 m; q_1q_2)\nonumber\\
 &&
\overline{\chi_1(q_1q_2'\overline{M_2r}( q_1q_2' r/n_1)^2 ) }
\chi_2(\pm \overline{q_1^2q_2'^2}  m' M_1) \overline{\tau^2(\chi_1 )}\;\overline{ \tau(\chi_2 )}\;
\overline{B( n_1, \pm v, \pm M_1^2 m'; q_1q_2')  }
e\left( \frac{\widetilde{n}_2 v}{ q_2q_2'q_1r/n_1}\right)
\bigg|\nonumber\\
&\leq& M_1^2M_2 | \mathcal{C}_2^*(\widetilde{n}_2) |,
\eea
where
\bna
| \mathcal{C}_2^*(\widetilde{n}_2) )|=\frac{n_1}{q_2q_2'q_1r}
\sum_{ v \bmod q_2q_2'q_1r/n_1}B( n_1, \pm v, \pm M_1^2 m; q_1q_2)
\overline{B( n_1, \pm v, \pm M_1^2 m'; q_1q_2')  }
e\left( \frac{\widetilde{n}_2 v}{ q_2q_2'q_1r/n_1}\right).
\ena
The character sum $ \mathcal{C}_2^*(\widetilde{n}_2)  $ can be estimated similarly as $ \mathcal{C}_2(\widetilde{n}_2)  $ in \eqref{C22}.
More precisely, we have the following estimate.
\begin{lemma}\label{mathcal{C} estimate}
We have
\bna
\mathcal{C}_2^*(\widetilde{n}_2)\ll
\mathop{\sum\sum}_{d_1, d_1'|q_1} d_1 d_1'
 \mathop{\sideset{}{^*}\sum_{u \bmod q_1r/n_1 \atop n_1u\equiv \pm \overline{M_2^2}M_1^2m\bmod d_1}\;
\sideset{}{^*}\sum_{u' \bmod q_1r/n_1 \atop
n_1u'\equiv \pm \overline{M_2^2}M_1^2m'\bmod d_1'}}_{
q_2\overline{u'} -q_2'\overline{u}\equiv \pm M_1\widetilde{n}_2\bmod q_1r/n_1}
\mathop{\sum\sum}_{d_2|(q_2, \overline{M_1^3}M_2^2 q_2'n_1\pm m \widetilde{n}_2)
\atop d_2'|(q_2', \overline{M_1^3}M_2^2 q_2n_1\mp m' \widetilde{n}_2)} d_2 d_2'.
\ena
Moreover, if $\widetilde{n}_2=0$, we have
\bna
q_2=q_2'
\ena
and
\bna
\mathcal{C}_2^*(0)\ll  q_1q_2r\mathop{\sum\sum}_{d, d'|q_1q_2 \atop (d, d')|(m-m')}(d, d').
\ena
\end{lemma}
By \eqref{Omega_22}, \eqref{character sum-222}, Lemma \ref{integral estimate2} and Lemma \ref{mathcal{C} estimate}, the contribution from the terms with $\widetilde{n}_2=0 $ to $ \widetilde{\mathbf{\Omega}}_2^{\pm}( n_1,q_1,r)$ is at most
\bna
&&\frac{M^{\varepsilon}  C}{Q}
\sum_{ q_2\sim C/q_1 \atop (q_2, q_1 M)=1}\frac{1}{q_2^3}
\sum_{m\sim N_2}\frac{1}{m^{1/2}}\,
\sum_{m'\sim N_2}\frac{|\lambda_g(m')|^2}{m'^{1/2}} M_1^2M_2 \cdot q_1q_2r
\mathop{\sum\sum}_{d, d'|q_1q_2 \atop (d, d')|(m-m')}(d, d')\\
\ll&&
\frac{M^{1+\varepsilon} M_1 C^2 r}{Q N_1}
\sum_{ q_2\sim C/q_1 \atop (q_2, q_1 M)=1}q_2^{-3}
\mathop{\sum\sum}_{ m,m'\sim N_2} |\lambda_g(m')|^2
\mathop{\sum\sum}_{d, d'|q_1q_2 \atop (d, d')|(m-m')}(d, d')\\
\ll&&
\frac{M^{1+\varepsilon} M_1 C^2 r}{Q  }
\sum_{ q_2\sim C/q_1 \atop (q_2, q_1 M)=1}q_2^{-3}
\sum_{ \ell |q_1q_2 } \ell \left(  \frac{N_1}{\ell}+1\right)\\
\ll&&
\frac{M^{1+\varepsilon} M_1 C^2 r}{Q  }
\left( \frac{q_1}{C}\right)^2( N_2+C)\\
\ll&&
\frac{M^{1+\varepsilon} M_1 q_1^2 r}{Q  }( N_2+C)
\ena
which is obviously dominated by the right hand of \eqref{compare 1}. It follows that the contribution from the terms with $ \widetilde{n}_2=0 $ to $ \mathbf{S}_2$ is smaller that the right hand side of \eqref{compare 1}.

By \eqref{Omega_22}- \eqref{character sum-222} and Lemma \ref{integral estimate2} - \ref{mathcal{C} estimate}, we have that the contribution from the terms with $\widetilde{n}_2 \neq 0 $ to $ \widetilde{\mathbf{\Omega}}_2^{\pm}( n_1,q_1,r)$ is bounded by
\bna
&&\frac{M^{1+\varepsilon} M_1 C }{Q N_1 }
\sum_{ q_2\sim C/q_1 \atop (q_2, q_1 M)=1}q_2^{-3/2}
 \sum_{ q_2'\sim C/q_1 \atop (q_2', q_1 M)=1} q_2'^{-3/2}
\sum_{m\sim N_2}\sum_{m'\sim N_2} |\lambda_g(m')|^2\\
\times &&
\sum_{ 0 \neq  \widetilde{n}_2 \ll L^{ **}}
\mathop{\sum\sum}_{d_1, d_1'|q_1} d_1 d_1'
 \mathop{\sideset{}{^*}\sum_{u \bmod q_1r/n_1 \atop n_1u\equiv \pm \overline{M_2^2}M_1^2m\bmod d_1}\;
\sideset{}{^*}\sum_{u' \bmod q_1r/n_1 \atop
n_1u'\equiv \pm \overline{M_2^2}M_1^2m'\bmod d_1'}}_{
q_2\overline{u'} -q_2'\overline{u}\equiv \pm M_1\widetilde{n}_2\bmod q_1r/n_1}
\mathop{\sum\sum}_{d_2|(q_2, \overline{M_1^3}M_2^2 q_2'n_1\pm m \widetilde{n}_2)
\atop d_2'|(q_2', \overline{M_1^3}M_2^2 q_2n_1\mp m' \widetilde{n}_2)} d_2 d_2'.
\ena
Recall $ L^{**}=Q^{1+\varepsilon}Cn_1r/(q_1L) < L^{*}$ in \eqref{widetilde{n}_2}. One sees that the above display is smaller than the right hand side of \eqref{nonzreo}. Therefore, the contribution from the terms with $ \widetilde{n}_2 \neq 0 $ to $ \mathbf{S}_2$ is smaller than the upper bound in \eqref{S1-nonzero}.

\section{Proof of Lemma \ref{GL2 lemma}}
In this section we apply $\rm GL_2$ Voronoi formula to transform $\mathscr{A}$.
By the Fourier expansion of $\chi$ in the terms of additive characters (see (3.12) in \cite{IK})
$$
\chi(m)=\frac{1}{\tau(\overline{\chi})}\sum_{c(\text{{\rm mod }} M)}
\overline{\chi}(c)e\left(\frac{cm}{M}\right),
$$
we have
\bea\label{A-sum}
\mathscr{A}&=&\sum_{m=1}^{\infty}\lambda_g(m)
e\left(\frac{am}{qM_1 }\right)
 V\left(\frac{m}{N}\right)
e\left(\frac{m\zeta}{qQM_1}\right)
 \frac{1}{\tau(\overline{\chi})}\sum_{c(\text{{\rm mod }} M)}
 \overline{\chi}(c)e\left(\frac{cm}{M}\right)
 \nonumber\\
&=&
\frac{1}{\tau(\overline{\chi})}\sum_{c(\text{{\rm mod }} M)}
\overline{\chi}(c)\sum_{m=1}^{\infty}\lambda_g(m)
e\left(\frac{a M_2+cq}{qM}m\right)
V\left(\frac{m}{N}\right)
e\left(\frac{m\zeta}{qQM_1}\right).
\eea

Let $f$ be a holomorphic cusp form of weight $k$ or a Maass cusp form with Laplace
eigenvalue $1/4+\mu^2$ for $\rm SL_2(\mathbf{Z})$. We have the following
Voronoi formula for $g$ (see \cite[Theorem A.4]{KMV}).

\begin{lemma}\label{GL2 Voronoi formula}
Let $\varphi(x)$ be a smooth function compactly supported on $\mathbf{R}^+$.
Let $a, \overline{a}, c\in\mathbf{Z}$ with $c\neq0, (a,c)=1$ and
$a\overline{a}\equiv1\;(\text{{\rm mod }} c)$. Then
\bna
\sum_{m=1}^{\infty}\lambda_g(m)e\left(\frac{am}{c}\right)\varphi\left(\frac{m}{N}\right)
=\frac{N}{c}\sum_\pm\sum_{m=1}^{\infty}\lambda_g(m)e\left(\pm\frac{\overline{a}m}{c}\right)
\Psi^\pm\left(\frac{mN}{c^2}\right),
\ena
where
\bea\label{GL2 integral-2}
\Psi^\pm(x)=\int_0^{\infty}\varphi(y)\mathbf{J}_g^{\pm}(4\pi\sqrt{xy})\mathrm{d}y,
\eea
with
\bna
\mathbf{J}_g^{+}(x)=2\pi i^k J_{k-1}(x), \qquad \mathbf{J}_g^{-}(x)=0\quad
\mbox{if}\; g\; \mbox{is holomorphic},
\ena
and
\bna
\mathbf{J}_g^{+}(x)=\frac{-\pi}{\sin \pi i\mu} \left(J_{2i\mu}(x)-J_{-2i \mu}(x)\right),
\qquad \mathbf{J}_g^{-}(x)=4\varepsilon_g\cosh(\pi \mu)K_{2i\mu}(x)\quad
\mbox{if}\; g\; \mbox{is Maass}.
\ena
Here $\varepsilon_g$ is the eigenvalue of $g$
under the reflection operator.
\end{lemma}

\begin{lemma}\label{voronoiGL2-Maass-asymptotic}
(1) For $0<x\ll 1$, we have
\bea\label{Bessel-small}
x^{\ell}\frac{\mathrm{d}^{\ell}}{\mathrm{d}x^\ell}\mathbf{J}_g^{+}(x)\ll_{\ell} 1,
\qquad x^{\ell}\frac{\mathrm{d}^{\ell}}{\mathrm{d}x^\ell}\mathbf{J}_g^{-}(x)\ll_{\ell} x^{-1/2}
\eea
for any integer $\ell\geq 0$.

(2) For any fixed integer $J\geq 1$ and $x\gg 1$, we have $\Psi^{-}(x)\ll_A x^{-A}$ for any $A>0$ and
\bea\label{Bessel-large}
\Psi^{+}(x)=x^{-1/4} \int_0^\infty \varphi(y)y^{-1/4}
\sum_{j=0}^{J}
\frac{c_{j} e(2 \sqrt{xy})+d_{j} e(-2 \sqrt{xy})}
{(xy)^{j/2}}\mathrm{d}y
+O_{\mu,J}\left(x^{-J/2-3/4}\right),
 \eea
where $c_{j}$ and $d_{j}$ are some constants depending on $\mu$.
\end{lemma}
\begin{proof}
(1) By the Taylor expansion of the $J$-Bessel function (see \cite[(8.402-8)]{GR}), for $|\arg z|<\pi$,
$$
J_{\nu}(x)=\frac{x^{\nu}}{2^{\nu}}\sum_{k=0}^{\infty}(-1)^k
\frac{x^{2k}}{2^{2k}k!\Gamma(\nu+k+1)}
$$
we have
$$
x^{\ell}J_{\nu}^{(\ell)}(x)\ll_{\nu,\ell} x^{\nu}
$$
for any integer $\ell\geq 0$. Thus the first inequality in \eqref{Bessel-small} follows.

Similarly, by the integral representation of $K_{\nu}(x)$ (see \cite[(8.432-8)]{GR}), for
$| \arg z|<\pi$, $\Re(\nu)>-1/2$, one has
$$
K_{\nu}(x)=\sqrt{\frac{\pi}{2x}}\frac{e^{-x}}{\Gamma(\nu+1/2)}
\int_0^{\infty}e^{-u}u^{\nu-1/2}\left(1+\frac{u}{2x}\right)^{\nu-1/2}\mathrm{d}u.
$$
Thus for $\nu\in i\mathbb{R}$ and $x>0$, we have
\bea\label{BoundOfK}
K_{\nu}(x)\ll_{\nu} x^{-1/2}e^{-x}.
\eea
In particular, for $0<x\ll 1$,
\bna
x^{\ell}K_{\nu}^{(\ell)}(x)\ll_{\nu,\ell} x^{-1/2}
\ena
for any integer $\ell\geq 0$, which yields the second inequality in \eqref{Bessel-small}.

\medskip

(2)
For $x\gg 1$, by \eqref{GL2 integral-2} and \eqref{BoundOfK}, it is easily seen that
$\Psi^{-}(x)\ll_A x^{-A}$ for any $A>0$. The asymptotic formula \eqref{Bessel-large} is proved in
Lemmas 3.3 and 3.4 in \cite{LS}.
\end{proof}

\medskip

Now we return to the sum over $m$ in \eqref{A-sum}.
Note that $(aM_2 +cq, q)=1$, $(aM_2+cq, M_2)=1$ and $(aM_2 +cq, M_1)=1$ or $M_1$. So
we distinguish two cases.

Case I. $(aM_2 +cq, M_1)=1$.

In this case $(aM_2 +cq, qM)=1$. Applying the $\rm GL_2$ Voronoi
formula in Lemma \ref{GL2 Voronoi formula} with
$\varphi(x)=V(x)e\left(\frac{\zeta N x}{qQM_1 }\right)$, we have
\bea\label{after GL2 Voronoi-1}
\sum_{m=1}^{\infty}\lambda_g(m)
e\left(\frac{(aM_2+cq)m}{qM}\right)
\varphi\left(\frac{m}{N}\right)
= \frac{N}{q M}
\sum_\pm\sum_{m=1}^{\infty}\lambda_g(m)
e\left(\pm\frac{\overline{aM_2+cq}}{qM}m\right)
\Psi^\pm\left(\frac{mN}{q^2M^2}\right),
\eea
where by \eqref{GL2 integral-2},
\bna
\Psi^\pm(x)=\int_0^{\infty}V(y)e\left(\frac{\zeta N y}{qQM_1 }\right)
\mathbf{J}_g^{\pm}(4\pi\sqrt{xy})\mathrm{d}y.
\ena

For $x\gg M^{\varepsilon}$, by the second statement of Lemma \ref{voronoiGL2-Maass-asymptotic},
we have $\Psi^-(x)\ll M^{-A}$ and the evaluation of $\Psi^{+}(x)$ is reduced to considering the integral
\bea\label{Psi integral}
\Psi_0(x,n,q)=x^{-1/4} \int_0^\infty y^{-1/4}V(y)
e\left(\frac{\zeta N y}{qQM_1 }\pm 2 \sqrt{xy}\right)\mathrm{d}y.
 \eea
Let $\varrho(y)=\frac{\zeta N y}{qQM_1 }\pm 2 \sqrt{xy}$. Then
$$
\varrho'(y)=\frac{\zeta N }{qQM_1 }\pm \sqrt{\frac{x}{y}}, \qquad
\varrho''(y)=\mp \frac{1}{2y}\sqrt{\frac{x}{y}}.
$$
Then repeated integration by parts shows that
the above integral is negligibly small unless we take the $-$ sign and $x\asymp \left(\frac{\zeta N }{qQM_1 }\right)^2$.
More precisely, by a stationary phase argument,
\bea\label{Psi integral-1}
\Psi_0(x,n,q)=x^{-1/2}F_{\natural}\left(\frac{q^2Q^2M_1^2x}{\zeta^2N^2}\right)
e\left(-\frac{qQM_1x}{\zeta N}\right)+O_A(M^{-A})
 \eea
 for any $A>0$.
So for $x=\frac{mN}{q^2M^2}$, we can restricted the sum over $m$ to
$$m \ll \frac{Q^2M^{2+\varepsilon}}{N}\ll M^{1+\varepsilon}M_2.$$
Here we recall that $Q^2M_1=N$.
On the other hand, for $x\ll M^{\varepsilon}$, the above restriction on $m$ holds trivially.
So in any case, up to a negligible error, we can restricted the sum over $m$ to
$m\leq M^{1+\varepsilon}M_2$.

By \eqref{A-sum}, \eqref{after GL2 Voronoi-1} and \eqref{Psi integral-1}, the contribution
from terms with $(aM+cq,qM)=1$ to $\mathscr{A}$ in \eqref{A-sum} is approximately
\bea\label{A-I}
\frac{ N^{1/2}}{\tau(\overline{\chi})}\sum_\pm
\sum_{c(\text{{\rm mod }} M)\atop c\not\equiv -\overline{q}aM_2\bmod M_1}
\overline{\chi}(c)
\sum_{m\leq M^{1+\varepsilon} M_2}\frac{\lambda_g(m)}{m^{1/2}}
e\left(\pm\frac{\overline{aM_2+cq}}{qM }m\right)
\mathfrak{I}^\pm\left(\frac{mN}{q^2 M^2 },q,\zeta\right)+O_A(M^{-A}),
\eea
where for $x\ll M^{\varepsilon}$,
\bea\label{I small}
\mathfrak{I}^\pm\left(x,q,\zeta\right)=x^{1/2}\int_0^{\infty}V(y)e\left(\frac{\zeta N y}{qQM_1 }\right)
\mathbf{J}_g^{\pm}(4\pi\sqrt{xy})\mathrm{d}y,
\eea
and for $x\gg M^{\varepsilon}$,
\bea\label{I large}
\mathfrak{I}^\pm\left(x,q,\zeta\right)=x^{1/4} \int_0^\infty y^{-1/4}V(y)
e\left(\frac{\zeta N y}{qQM_1 }\pm 2 \sqrt{xy}\right)\mathrm{d}y.
\eea

Case II. $(aM_2+cq,M_1)=M_1$.

In this case, $(aM_2 +cq, qM_1)=M_1$ and $c\equiv -\overline{q}aM_2$ mod $M_1$.
As in Case I, we apply Lemma \ref{GL2 Voronoi formula}
with $\varphi(x)=V(x)e\left(\zeta Nx/(qQ M_1)\right)$ to get that $m$-sum in \eqref{A-sum} equals
\bea \label{case ii}
&&\sum_m\lambda_g(m)e\left({\frac{(aM_2+cq)/M_1}{qM_2}m}\right)\varphi\left(\frac{m}{N}\right)\nonumber\\
&=&\frac{N}{qM_2}\sum_\pm\sum_m \lambda_g(m)
e\left(\pm\frac{\overline{(aM_2+cq)/M_1}}{qM_2}m\right)
\Psi^{\pm}\left(\frac{mN}{q^2M_2^2}\right),
\eea
where $\Psi^{\pm}(x)$ is defined in \eqref{GL2 integral-2}.
By the argument in Case I, we have, up to a negligible error, the sum over $m$ can be restricted by
\bna
\frac{mN}{q^2M_2^2}\cdot\frac{q^2Q^2M_1^2}{N^2}\leq M^{\varepsilon},
\ena
i.e., $m\leq M_2^{2+\epsilon}/M_1$.
For these values of $m$, by \eqref{Psi integral},
\bea \label{Psi-case ii}
\Psi^\pm\left(\frac{mN}{q^2M_2^2}\right)=
\frac{qM_2}{N^{1/2}m^{1/2}}
\mathfrak{I}^\pm\left(\frac{mN}{q^2M_2^2}, q,\zeta\right),
\eea
where $\mathfrak{I}^\pm(x,q,\zeta)$ is defined in \eqref{I small} and \eqref{I large}.

By \eqref{A-sum}, \eqref{case ii} and \eqref{Psi-case ii},
the contribution from $(aM_2+cq,qM)=M_1$ to $\mathscr{A}$ in \eqref{A-sum} equals
\bea\label{A-II}
\frac{N^{\frac{1}{2}}}{\tau(\overline{\chi})}
\sum_\pm
\sum_{c(\bmod M)\atop c\equiv-\overline{q}aM_2 \bmod M_1}\overline{\chi}(c)
\sum_{m\leq M_2^{2+\epsilon}/M_1}
\frac{\lambda_g(m)}{m^{1/2}}
e\left(\pm\frac{\overline{(aM_2+cq)/M_1}}{qM_2}m\right)
\mathfrak{I}^\pm\left(\frac{mN}{q^2M_2^2}, q,\zeta\right).
\eea

Then Lemma \ref{GL2 lemma} follows from \eqref{A-I}, \eqref{I small}, \eqref{I large} and \eqref{A-II}.

\section{Proof of Lemma \ref{GL3 lemma}}
\setcounter{equation}{0}
\medskip
In this section we will apply the $\rm GL_3$ Voronoi formula to transform $\mathscr{B}$,
where
\bna
\mathscr{B}=
\sum_{n}\lambda_{\pi}(n,r) e\left(-\frac{an}{qM_1}\right)\phi(n),
\ena
where $\phi(x)=W(x/N)e\left(-\zeta x/(qQM_1) \right)$.

We now recall the Voronoi summation formula for $\rm SL_3(\mathbf{Z})$.
For $\phi(x)\in C_c^\infty(0,\infty)$ we
denote by $\widetilde{\phi}(s)=\int_0^{\infty}\phi(x)x^{s-1}\mathrm{d}x$
the Mellin transform of $\phi(x)$.
Let
\bea\label{intgeral transform-3}
\Phi{\pm}\left(x\right)
=\frac{1}{2\pi i}\int_{(\sigma)}x^{-s}
\gamma_{\pm}(s)\widetilde{\phi}(-s)\mathrm{d}s,\qquad
\sigma>\max\limits_{1\leq j\leq 3}\{-1-\mathrm{Re}(\mu_j)\},
\eea
with
\bna
\gamma_{\pm}(s)=\frac{1}{2\pi^{3(s+1/2)}}\left(\prod_{j=1}^3
\frac{\Gamma\left((1+s+\mu_j)/2\right)}
{\Gamma\left((-s-\mu_j)/2\right)}\mp i\prod_{j=1}^3
\frac{\Gamma\left((2+s+\mu_j)/2\right)}
{\Gamma\left((-s-\mu_j+1)/2\right)}\right),
\ena
where $\mu_j$, $j=1,2,3$, are the Langlands parameters of $\pi$.
Then we have the following Voronoi summation formula (see \cite{GL1}, \cite{MS}).

\begin{lemma}\label{voronoiGL3}
Let $q\in \mathbb{N}$ and $a\in \mathbf{Z}$ be such
that $(a,q)=1$. Then
\bna
\sum_{n=1}^{\infty}\lambda_{\pi}\left(n,r\right)e\left(\frac{an}{q}\right)
\phi\left(n\right)= q\sum_{\pm}\sum_{n_{1}|qr}
\sum_{n_{2}=1}^{\infty}\frac{\lambda_{\pi}\left(n_{1},n_{2}\right)}{n_{1}n_{2}}
S\left(r\overline{a},\pm n_{2};\frac{qr}{n_{1}}\right)
\Phi_{\pm}\left(\frac{n_{1}^{2}n_{2}}{q^{3}r}\right),
\ena
where $a \overline{a} \equiv 1(\bmod q)$ and $S(m,n;c)$ is the classical Kloosterman sum.
\end{lemma}

Applying the $\rm GL_3$ Voronoi formula in Lemma \ref{voronoiGL3}, we obtain
\bea\label{GL3 after Voronoi}
\mathscr{B}=qM_1\sum_{\pm}\sum_{n_{1}|qM_1 r} \sum_{n_2=1}^{\infty} \frac{\lambda_{\pi}(n_1,n_2)}{n_{1}n_{2}}
  S\left(-r\overline{a},\pm n_{2};\frac{qM_1  r}{n_{1}}\right)
  \Phi_{\pm}\left(\frac{n_{1}^{2}n_{2}}{q^3M_1^3 r}\right),
\eea
where $\Phi_{\pm}(x)$ is defined as in \eqref{intgeral transform-3}.

First, we study the integral transform  $\Phi_{\pm}(x)$ in \eqref{intgeral transform-3}. By Stirling's formula, for $\sigma\geq -1/2$,
\bna
\gamma_{\pm}(\sigma+i\tau)\ll_{\pi,\sigma}(1+|\tau|)^{3\left(\sigma+1/2\right)}.
\ena
Moreover, for $s=\sigma+i\tau$,
$$
\widetilde{\phi}(-s)=N^{-s} W^{\dag}\left(\frac{\zeta N}{qQM_1 },-s\right)\ll
N^{-\sigma}\min \left\{1,\left(\frac{N}{qQM_1|\tau|}\right)^{j}\right\}
$$
for any $j\geq 0$. Thus
\bna
\Phi_{\pm}\left(x\right)
&\ll&
x^{-\sigma}N^{-\sigma}
 \int_{\mathbb{R}}(1+|\tau|)^{3(\sigma+1/2)}
\min\left\{1, \left(\frac{N}{qQM_1|\tau|}\right)^j\right\}
\mathrm{d}\tau
\\
&\ll&
(x N)^{-\sigma}\left(\frac{N}{qQM_1}\right)^{3\sigma+5/2}
\\
&\ll&
\left(\frac{N}{qQM_1 }\right)^{5/2}
\left(\frac{Q^3q^3M_1^3x}{N^2}\right)^{-\sigma}.
\ena
Thus,
$\Phi_{\pm}\left(\frac{n_{1}^{2}n_{2}}{q^3M_1^3 r}\right)$ on the right
hand side of \eqref{GL3 after Voronoi} gives arbitrary power savings in $M$ if
$n_1^2n_2> N^{2+\varepsilon}r/Q^3$ for any $\varepsilon>0$.
For small values of $n_1^2n_2$, we move the integration line to
$\sigma=-1/2$ to get
\bea\label{special value}
\Phi_{\pm}\left(\frac{n_{1}^{2}n_{2}}{q^3M_1^3 r}\right)
&=&
\frac{1}{2\pi}\int_{\mathbb{R}}
\left(\frac{n_1^2n_2}{q^3M_1^3 r}\right)^{\frac{1}{2}-i\tau}
\gamma_{\pm}
\left(-\frac{1}{2}+i\tau\right)
N^{\frac{1}{2}-i\tau}W^{\dag}\left(\frac{\zeta N}{qQM_1}, \frac{1}{2}-i\tau\right)\mathrm{d}\tau
\nonumber\\
&:=&
\left(\frac{Nn_{1}^{2}n_{2}}{q^3M_1^3 r}\right)^{1/2}\mathfrak{J}^{\pm}
\left(\frac{n_{1}^{2}n_{2}}{q^3M_1^3 r},q,\zeta\right),
\eea
where
\bea\label{J-integral definition}
\mathfrak{J}^{\pm}
\left(x,q,\zeta\right)=\frac{1}{2\pi}
\int_{\mathbb{R}}(Nx)^{-i\tau}
\gamma_{\pm}\left(-\frac{1}{2}+i\tau\right) W^{\dag}\left(\frac{\zeta N}{qQM_1},\frac{1}{2}-i\tau\right)\mathrm{d}\tau.
\eea
By \eqref{GL3 after Voronoi} and \eqref{special value},
\bea\label{B expression}
\mathscr{B}&=&\frac{N^{1/2}}{q^{1/2}M_1^{1/2} r^{1/2}}\sum_{\pm}\sum_{n_{1}|qM_1  r}\;
\sum_{n_1^2n_2\leq \frac{N^{2+\varepsilon}r}{Q^3}}
\frac{\lambda_{\pi}(n_1,n_2)}{n_2^{1/2}}
  S\left(-r\overline{a},\pm n_{2};\frac{qM_1 r}{n_{1}}\right)\nonumber\\&&\qquad\qquad\qquad\qquad
  \qquad\qquad\qquad
\times\mathfrak{J}^{\pm}\left(\frac{n_{1}^{2}n_{2}}{q^3M_1^3 r},q,\zeta\right)+O(M^{-A}).
\eea
Then Lemma \ref{GL3 lemma} follows from \eqref{J-integral definition} and \eqref{B expression}.


\begin{thebibliography}{100}


\bibitem{B}{}
V. Blomer, {\it Subconvexity for twisted $L$-functions on GL(3)},
Amer. J. Math. 134 (2012), no. 5, 1385-1421.



\bibitem{DFI}{}

W. Duke, J. B. Friedlander, H. Iwaniec,
{\it Bounds for automorphic L -functions},
Invent. Math. 112 (1993), no. 1, 1-8.



\bibitem{GL1}{}
D. Goldfeld and X. Li,
{\it Voronoi formulas on $GL(n)$},
Int. Math. Res. Not. 2006, Art. ID 86295, 25 pp.

\bibitem{GR}{}
I. S. Gradshteyn and I. M. Ryzhik,
{\it Table of integrals, series, and products}, edition 7,
Translated from the Russian;
Translation edited and with a preface by Alan Jeffrey and Daniel Zwillinger;
With one CD-ROM (Windows, Macintosh and UNIX),
Elsevier/Academic Press, Amsterdam, 2007,
xlviii+1171 pages.
	
\bibitem{HB}{}
Huang, Bingrong,
{\it On the Rankin-Selberg problem},
Math. Ann. 381 (2021),  no. 3-4, 1217-1251.


\bibitem{Hux2}{}
M. N. Huxley, {\it Area, lattice points, and exponential sums},
London Mathematical Society Monographs. New Series,
volume 13, Oxford Science Publications,
The Clarendon Press, Oxford University Press, New York,
1996, xii+494 pages.


\bibitem{IK}{}
H. Iwaniec, E, Kowalski, {\it Analytic number theory},
Amercian Mathematical Society Colloquium Publications 53, Amercian Mathematical Society, Providence, RI, 2004.




\bibitem{KMV}{}
E. Kowalski, P. Michel and J. Vanderkam.
{\it Rankin-Selberg $L$-functions in the level aspect},
Duke Mathe. J. {\bf 114} (2002), 123-191.




\bibitem{LMS}{}
 Y. Lin, P. Michel, W. Sawin,
{\it Algebraic twists of $\rm GL_3 \times \rm GL_2$ $L$-functions},
arXiv:1912.09473.

\bibitem{LS}{}
 Y. Lin, Q. Sun,
{\it Analytic twists of $GL(3)\times GL(2)$ automorphic forms},
Int. Math. Res. Not. IMRN 2021(19) (2021),15143-15208.


\bibitem{Mol}{}
G. Molteni,
{\it Upper and lower bounds at $s=1$ for certain Dirichlet series with Euler product},
Duke Mathematical Journal 111 (2002), no. 1, 133-158.

\bibitem{Munshi1}{}
R. Munshi,
{\it The circle method and bounds for $L$-functions, II:
Subconvexity for twists of $GL(3)$ $L$-functions},
Amer. J. Math. 137 (2015), no. 3, 791-812.



\bibitem{Munshi2}{}
R. Munshi,
{\it The circle method and bounds for $L$-functions-III:
$t$-aspect subconvexity for $GL(3)$ $L$-functions},
J. Amer. Math. Soc. 28 (2015), no. 4, 913-938.



\bibitem{Munshi2018}{}
R. Munshi,
{\it Subconvexity for $\rm GL(3)\times \rm GL(2)$ $L$-functions in $t$-aspect},
\emph{ArXiv preprint} (2018), arXiv:1810.00539.


\bibitem{MS}{}
S.D. Miller and W. Schmid,
{\it Automorphic distributions, $L$-functions, and Voronoi summation for $GL(3)$},
Ann. of Math. (2) 164 (2006), no. 2, 423-488.


\bibitem{Sharma2019}
P. Sharma,
{\it Subconvexity for $\rm GL(3)\times \rm GL(2)$ twists in level aspect,}
\emph{ArXiv preprint} (2019), arXiv:1906.09493.



\end{thebibliography}
\end{document}